\documentclass{amsart} 

\usepackage{euscript}           
\usepackage{pstricks,color}
\usepackage{pst-plot}
\usepackage{float,lscape}
\input{multido.tex}

\usepackage{amsrefs,amsmath,amsthm,amssymb}            
\usepackage{graphicx,enumerate,calc,lscape}
\usepackage[matrix,arrow,curve,frame]{xy}   

\usepackage{longtable,booktabs}

\newtheorem{thm}[subsection]{Theorem}
\newtheorem{prop}[subsection]{Proposition}

\newtheorem{lemma}[subsection]{Lemma}

\theoremstyle{definition}  
\newtheorem{defn}[subsection]{Definition}

\newtheorem{example}[subsection]{Example}
\newtheorem{remark}[subsection]{Remark}

\newcommand{\cat}{\EuScript}    
\newcommand{\cA}{{\cat A}}      
\newcommand{\cB}{{\cat B}}      
\newcommand{\sA}{\cA}

\newcommand{\cobar}{{\mathcal{C}}}

\newcommand{\ra}{\rightarrow}                   
\newcommand{\map}{\rightarrow}
\newcommand{\ol}{\overline}
\newcommand{\epsilonbar}{\ol{\epsilon}}
\newcommand{\cl}{\mathrm{cl}}
\newcommand{\tp}{\mathrm{top}}

\newcommand{\field}[1]  {\mathbb #1} 
\newcommand{\F}         {\field F}
\newcommand{\R}         {\field R}
\newcommand{\Z}         {\field Z}
\newcommand{\C}         {\field C}
\newcommand{\M}         {\field M}
\newcommand{\A}         {\field A}

\DeclareMathOperator{\Gr}{Gr}
\DeclareMathOperator{\Sq}{Sq}

\DeclareMathOperator{\Hom}{Hom}
\DeclareMathOperator{\Ext}{Ext}
\DeclareMathOperator{\alg}{alg}

\newcommand{\dfn}{\textbf} 
\newcommand{\mdfn}[1]{\dfn{\mathversion{bold}#1}} 

\newcommand{\Smash}             {\wedge}
\newcommand{\tens}              {\otimes}               
\newcommand{\iso}               {\cong}

\numberwithin{equation}{subsection}

\newenvironment{myequation}
  {\addtocounter{subsection}{1}\begin{eqnarray}}
  {\end{eqnarray}$\!\!$}
  
\newcounter{themyfigure}

\newenvironment{myfigure}
	{\refstepcounter{themyfigure}}
    {}

\newrgbcolor{pcolor}{1 0 0}
\newrgbcolor{hzerocolor}{0.1 0.7 0.1}
\newrgbcolor{honecolor}{0 0 1}
\newrgbcolor{MWetacolor}{0 0 1}
\newrgbcolor{MWrhocolor}{1 0 0}
\newgray{gridline}{0.9}

\newcommand{\cirrad}{0.08}
\newcommand{\crad}{0.15}
\newcommand{\bigcrad}{0.2}
\newcommand{\labelrad}{0.1}

\newrgbcolor{tauzerocolor}{0.5 0.5 0.5}
\newrgbcolor{tauonecolor}{1 0 0}
\newrgbcolor{hzerotaucolor}{1 0 1}
\newrgbcolor{hzerotowercolor}{0.5 0.5 0.5}
\newrgbcolor{honetowercolor}{1 0 0}
\newrgbcolor{htwotaucolor}{1 0 1}

\begin{document}

\title{Low dimensional Milnor-Witt stems over $\R$}

\author{Daniel Dugger}
\address{Department of Mathematics\\ University of Oregon\\ Eugene, OR
97403} 
\email{ddugger@math.uoregon.edu}

\author{Daniel C.\ Isaksen}
\address{Department of Mathematics \\ Wayne State University\\
Detroit, MI 48202}
\email{isaksen@wayne.edu}




\begin{abstract}
This article
computes some motivic stable homotopy groups over $\R$.
For $0 \leq p - q \leq 3$,
we describe the motivic stable homotopy groups
$\hat\pi_{p,q}$ of a completion of the motivic sphere spectrum.  
These are the first four
Milnor-Witt stems.  We start with the known $\Ext$ groups over $\C$
and apply the $\rho$-Bockstein spectral
sequence to obtain $\Ext$ groups over $\R$.
This is the input to an Adams spectral sequence, which collapses in our
low-dimensional range.
\end{abstract}

\maketitle


\section{Introduction}
This paper takes place in the context of motivic stable homotopy
theory over $\R$.  
Write $\M_2=H^{*,*}(\R;\F_2)$ for the bigraded motivic
cohomology ring of a point, and write $\sA$ for the motivic Steenrod
algebra at the prime $2$.  
Our goal is to study the tri-graded Adams spectral sequence
\[ E_2=\Ext_\sA^{*,*,*}(\M_2,\M_2) \Rightarrow \hat\pi_{*,*},
\]
where $\hat\pi_{*,*}$ represents the stable motivic
homotopy groups of a completion of the motivic sphere spectrum over $\R$.
Specifically, in a range of dimensions we
\begin{enumerate}[(1)]
\item Compute the $\Ext$ groups appearing in the $E_2$-page
of the motivic Adams spectral sequence over $\R$.
\item Analyze all Adams differentials.
\item Reconstruct the groups $\hat\pi_{*,*}$ from their
filtration quotients given by the Adams $E_\infty$-page.
\end{enumerate}
Point (1) is tackled by introducing an auxilliary, purely algebraic
spectral sequence that converges to these $\Ext$ groups.  

To describe our results more specifically we must
 introduce some notation and terminology related to the three
indices in our spectral sequence.  We first have the homological
degree of the $\Ext$ groups, also called the Adams filtration
degree---we label this $f$ and simply call it the {\it filtration\/}.
We then have the internal bidegree $(t,w)$ for $\sA$-modules, where
$t$ is the usual {\it topological degree\/} and $w$ is the {\it
weight\/}.  We introduce the grading $s=t-f$ and call this the {\it
topological stem\/}, or just the {\it stem\/}.  The triple $(s,f,w)$
of stem, filtration, and weight will be our main index of
reference.  Using these variables, the motivic Adams spectral sequence 
can be written
\[ E_2=\Ext_\sA^{s,f,w}(\M_2,\M_2) \Rightarrow \hat\pi_{s,w}.
\]

\vspace{0.1in}

Morel has computed
\cite{Morel05}
that $\hat\pi_{s,w}=0$ for $s<w$ 
(in fact, this is true integrally before completion).
Write
$\Pi_0=\oplus_n \pi_{n,n}$, considered as a $\Z$-graded ring.  This
is called the {\it Milnor-Witt
ring\/}, and Morel has given a complete description of this via
generators and relations \cite{Morel04b}.
It is convenent to set $\Pi_k=\oplus_n
\hat\pi_{n+k,n}$, as this is a $\Z$-graded module over $\Pi_0$.  We
call $\Pi_k$ the completed \dfn{Milnor-Witt $k$-stem}.  Related to this, 
the group $\hat\pi_{s,w}$ has  {\it Milnor-Witt degree\/} $s-w$.    

The completed Milnor-Witt ring $\Pi_0$ is equal to 
\[ 
\Z_2[\rho,\eta]/(\eta^2\rho + 2\eta, \rho^2\eta+2\rho),
\]
where $\eta$ has degree $(1,1)$ and $\rho$ has degree $(-1,-1)$.  
Note that $\Pi_0$ is the 2-completion of the Milnor-Witt ring of $\R$
described by Morel \cite{Morel04b}.

We have found that 
the analysis of the motivic Ext groups over $\R$, and of the Adams spectral
sequence, is most conveniently 
done with respect to the Milnor-Witt degree.  In
this paper we focus only on the range $s-w\leq 3$, leading to an
analysis of the Milnor-Witt stems $\Pi_1$,
$\Pi_2$, and $\Pi_3$.  The restriction to $s-w\leq 3$ is done for
didactic purposes; our methods can be applied to cover a much greater
range, but at the expense of more laborious computation.  The focus on
$s-w\leq 3$ allows us to demonstrate the methods and see examples of the
interesting phenomena, while keeping the intensity of the labor down
to manageable levels.  

\vspace{0.2in}

\subsection{An algebraic spectral sequence for {\mdfn{$\Ext$}}}
The main tool in this paper is the $\rho$-Bockstein spectral sequence
that computes the groups $\Ext_\cA(\M_2,\M_2)$.  This was originally
introduced by Hill \cite{Hill11} and analyzed for 
the subalgebra $\cA(1)$ of $\cA$ generated by $\Sq^1$ and $\Sq^2$.
Most of our hard work
is focused on analyzing the differentials in this spectral sequence,
as well as the hidden extensions encountered when passing from the
$E_\infty$-page to the true $\Ext$ groups.  

Over the ground field $\R$, one has $\M_2=\F_2[\tau,\rho]$
where $\tau$ has bidegree $(0,1)$ and $\rho$ has bidegree $(1,1)$.  In
contrast, over $\C$ one has $\M_2^\C=\F_2[\tau]$.  Write $\sA^\C$ for
the motivic Steenrod algebra over $\C$.  The groups 
$\Ext_{\sA^\C}^{s,f,w}(\M_2^\C,\M_2^\C)$ were computed for
$s\leq 34$ in \cite{DI10}, and then for $s\leq 70$ in \cite{Isaksen14c}.  
The $\rho$-Bockstein spectral sequence takes these groups as input,
having the form
\begin{myequation}
\label{eq:rho-Bockstein} 
E_1=\Ext_{\sA^\C}(\M_2^\C,\M_2^\C)[\rho] \Rightarrow
\Ext_\sA(\M_2,\M_2).
\end{myequation}
The differentials in this spectral sequence are extensive.  
However, in a large range they can be completely analyzed by
a method we describe next.

As an $\F_2[\rho]$-module,
$\Ext_\sA(\M_2,\M_2)$ 
splits as a summand of $\rho$-torsion modules and
$\rho$-non-torsion modules; we call the latter {\it $\rho$-local\/}
modules for short.  The first step in our work is to analyze the
$\rho$-local part of the $\Ext$ groups, and this turns out to have a
remarkably simple answer.  We prove that
\[ 
\Ext_\sA(\M_2,\M_2) [\rho^{-1}] \iso \Ext_{\sA_{cl}}(\F_2,\F_2)[\rho^{\pm 1}]
\]
where $\sA_{cl}$ is the classical Steenrod algebra at the prime $2$.  
The isomorphism is highly structured, in the sense that it is
compatible with all products and Massey products, and 
the element $h_i$ in $\Ext_{\sA_{cl}}(\F_2,\F_2)$ corresponds to the element
$h_{i+1}$
in $\Ext_{\sA}(\M_2,\M_2)[\rho^{-1}]$
for every $i\geq 0$.
In other words, the motivic
$\Ext$ groups $\Ext_{\sA}(\M_2,\M_2)$ have a shifted copy of
$\Ext_{\sA_{\cl}}(\F_2,\F_2)$ sitting inside them as the $\rho$-local part.

It turns out that through a large range of dimensions,
there is only one pattern of $\rho$-Bockstein differentials that is consistent
with the $\rho$-local calculation described in the previous paragraph.
This is what allows the analysis of the $\rho$-Bockstein
spectral sequence (\ref{eq:rho-Bockstein}).  It is not so easy to
organize this calculation: the tri-graded nature of the
spectral sequence, coupled with a fairly irregular pattern of
differentials, makes it close to impossible to depict the spectral
sequence via the usual charts.  We analyze what is happening via a
collection of charts and tables, but mostly focusing on the tables.  
A large portion of the present paper is devoted to explaining how to
navigate this computation.  

Figure \ref{fig:ExtR} shows the result
(all figures are contained in Section \ref{sctn:chart}).
This figure displays $\Ext_\cA(\M_2,\M_2)$ through Milnor-Witt degree 4.
Our computations agree with machine computations carried out by
Glen Wilson and Knight Fu \cite{WF}.

\subsection{Adams differentials}
Once we have computed $\Ext_\sA(\M_2,\M_2)$, the next step is the
analysis of Adams differentials.  Identifying even {\it possible\/}
differentials is again hampered by the
tri-graded nature of the situation, but we explain the calculus
that allows one to accomplish this---it is not as easy as it is for
the classical Adams spectral sequence, but it is at least mechanical.
In the range $s-w\leq 3$ there are only a few possible
differentials for degree reasons.  We show via
some Toda bracket arguments that in fact all of the differentials
are zero.

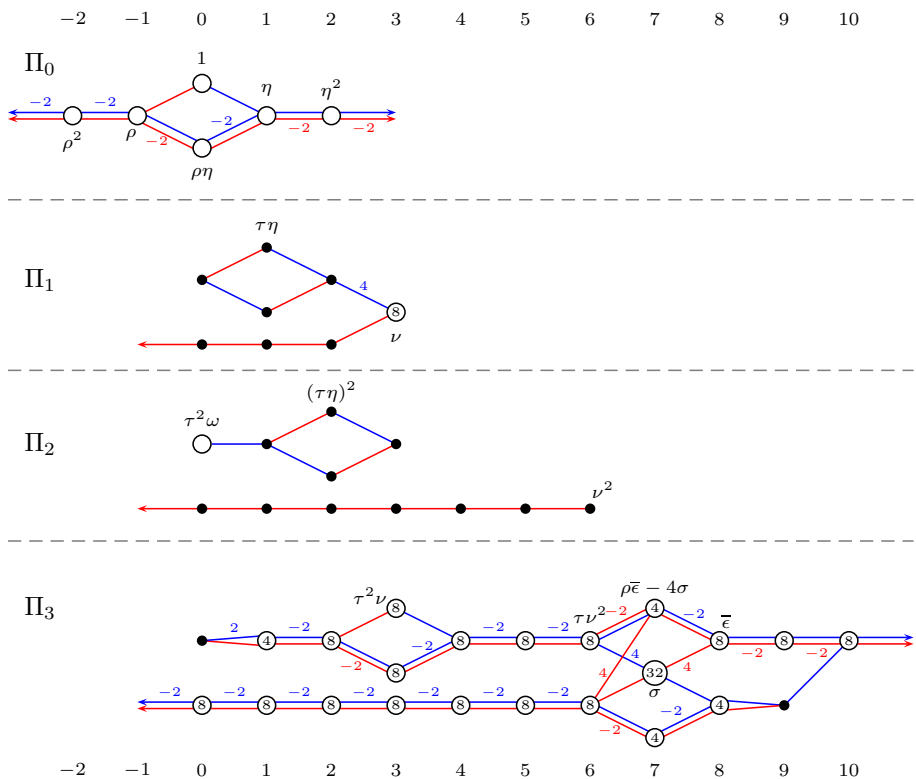
\subsection{Milnor-Witt modules}

After analyzing Adams differentials, we obtain the Adams $E_\infty$-page,
which is an associated graded object of the motivic stable 
homotopy groups over $\R$.  We convert the associated graded information
into the structure of the Milnor-Witt modules $\Pi_1$,
$\Pi_2$, and $\Pi_3$, as modules over $\Pi_0$.
We must be wary of extensions that are hidden by the 
Adams spectral sequence, but these turn out to be manageable.

Figure \ref{fig:MW-module} describes the results of this process.
We draw attention to a curious phenomenon in the 7-stem of
$\Pi_3$.  Here we see that the third Hopf map $\sigma$ has order 32,
not order 16.  This indicates that the motivic image of $J$
is not the same as the classical image of $J$. 
This unexpected behavior suggests that the theory 
of motivic (and perhaps equivariant)
$v_1$-self maps is not what one might expect.
These phenomena deserve more study.

We also observe that the 1-stem of $\Pi_1$ is consistent
with Morel's conjecture on the structure of
$\pi_{1,0}$.  (See \cite{OO14}*{p.\ 98} for a clearly stated
version of the conjecture.)

Unsurprisingly, our calculations are similar to calculations
of $\Z/2$-equivariant stable homotopy groups \cite{AI}.
There is a realization functor from motivic homotopy theory
over $\R$ to $\Z/2$-equivariant homotopy theory, and this
functor induces an isomorphism in stable homotopy groups in a range.
We will return to this comparison in future work.

\subsection{Other base fields}

Although we only work with the base field $\R$ in this article,
the phenomena that we study most likely occur for other base fields as well.
This is especially true for fields $k$ that are similar to $\R$,
such as fields that have an embedding into $\R$.

One might use our calculations to speculate on the
structure of $\Pi_1$, $\Pi_2$, and $\Pi_3$ for arbitrary base fields.
We leave this to the imagination of the reader.

\subsection{Organization of the paper} 

We begin in Section \ref{sctn:back} with a brief reminder of 
the motivic Steenrod algebra and the motivic Adams spectral sequence.
We construct the $\rho$-Bockstein spectral sequence
in Section \ref{sctn:p-Bockstein:generalities}, and 
we perform some preliminary calculations.
In Section \ref{sctn:p-local}, we consider the effect of inverting $\rho$.
Then we return in Section \ref{sctn:Bockstein-analysis} to a detailed
analysis of the $\rho$-Bockstein spectral sequence.  
We resolve extensions that are hidden in the $\rho$-Bockstein spectral
sequence in Section \ref{sctn:ExtR}, and obtain a description
of $\Ext_{\cA}(\M_2, \M_2)$.
We show that there are no Adams differentials in Section \ref{sctn:Adams}.
In Section \ref{sctn:MW-module}, we convert the associated graded information
of the Adams spectral sequence into explicit descriptions of 
Milnor-Witt modules.
Sections \ref{sctn:table} and \ref{sctn:chart}
contain the tables and charts required to carry out our detailed
computations.  We have collected this essential information in one place for the
convenience of readers who are seeking specific computational facts.

\subsection{Notation}

For the reader's convenience, we provide a table of notation to be used
later.

\begin{enumerate}
\item
$\M_2 = \F_2 [\tau, \rho]$ is the motivic $\F_2$-cohomology ring of $\R$.
\item
$\M_2^\C = \F_2 [\tau$ is the motivic $\F_2$-cohomology ring of $\C$.
\item
$\sA$ is the motivic Steenrod algebra over $\R$ at the prime $2$.
\item
$\sA_*$ is the dual motivic Steenrod alegbra over $\R$ at the prime $2$.
\item
$\sA^\C$ is the motivic Steenrod algebra over $\C$ at the prime $2$.
\item
$\sA_{\cl}$ is the classical Steenrod algebra at the prime $2$.
\item
$\Ext$ or $\Ext_\R$ is the cohomology of $\sA$, i.e.,
$\Ext_{\sA}(\M_2,\M_2)$.
\item
$\Ext_\C$ is the cohomology of $\sA^\C$, i.e.,
$\Ext_{\sA^\C}(\M_2^\C,\M_2^\C)$.
\item
$\Ext_{\cl}$ is the cohomology of $\sA^{\cl}$, i.e.,
$\Ext_{\sA^{\cl}}(\F_2,\F_2)$.
\item
$\hat\pi_{*,*}$ is the bigraded stable homotopy ring of the 
completion of the motivic sphere spectrum over $\R$ with respect
to the motivic Eilenberg-Mac Lane spectrum $H \F_2$.
\item
$\Pi_k = \oplus_n \hat\pi_{n+k,k}$
is the $k$th completed Milnor-Witt stem over $\R$.
\end{enumerate}

\section{Background}
\label{sctn:back}

This section establishes the basic setting and notation that will be
assumed throughout the paper.

Write
$\M_2=H^{*,*}(\R;\F_2)$ for the (bigraded) motivic cohomology
ring of $\R$.  We use the usual motivic bigrading where the first
index is the topological dimension and the second index is the weight.
Recall that $\M_2$ 
is equal to $\F_2[\tau, \rho]$, where $\tau$ has degree $(0,1)$
and $\rho$ has degree $(1,1)$. 
The class $\rho$ is the element $[-1]$ under the
standard isomorphism $\M_2^{1,1}\iso F^*/(F^*)^2$, and $\tau$ is the unique
element such that $\Sq^1(\tau) = \rho$.

Let $\cA_*$ denote the dual motivic Steenrod
algebra over $\R$.  The pair $(\M_2,\cA_*)$ is a Hopf algebroid; recall 
from \cite{Voevodsky03a} (see also \cite{Borghesi07})
that this structure is described by
\begin{align*}
& \cA_* = \M_2 [ \tau_0, \tau_1, \ldots, \xi_0, \xi_1, \ldots ] /
  ( \xi_0 = 1, 
  \tau_k^2 = \tau \xi_{k+1} + \rho \tau_{k+1} + \rho \tau_0 \xi_{k+1}) \\
& \eta_L(\tau) = \tau,\quad
 \eta_R(\tau) = \tau + \rho \tau_0,\quad \eta_L(\rho)=\eta_R(\rho)=\rho \\
& \Delta( \tau_k ) = \tau_k \otimes 1 +
  \sum \xi_{k-i}^{2^i} \otimes \tau_i \\
& \Delta( \xi_k ) = \sum \xi_{k-i}^{2^i} \otimes \xi_i.
\end{align*}
The Hopf algebroid axioms force $\Delta(\tau)=\tau\tens 1$ and
$\Delta(\rho)=\rho\tens 1$, but it is useful to record these for
reference.
The dual $\cA_*$ is homologically graded, so 
$\tau$ has degree $(0,-1)$ and $\rho$ has degree $(-1,-1)$.
Moreover, 
$\tau_k$ has degree $(2^{i+1} - 1, 2^i - 1)$
and $\xi_k$ has degree $(2^{i+1} - 2, 2^i - 1)$.

The groups $\Ext_{\cA_*}(\M_2,\M_2)$ are trigraded.  There is the
homological degree $f$ (the degree on the $\Ext$) and the internal
bidegree $(p,q)$ of $\cA_*$-comodules.  The symbol $f$ comes from
`filtration', as this index coincides with the Adams filtration in the
Adams spectral sequence.  Classical notation would write $\Ext^{f,(p,q)}$ for the
corresponding homogeneous piece of the $\Ext$ group.  In the Adams
spectral sequence this $\Ext$ group contribues to $\pi_{p-f,q}$.  We
call $p-f$ the \dfn{stem} and will usually denote it by $s$.  It turns
out to be more
convenient to use the indices $(s,f,w)$ of stem, filtration, and
weight rather than $(f,p,q)$.  So we will write $\Ext^{s,f,w}$ for the
group that would classically be denoted $\Ext^{f,(s+f,w)}$.  This
works very well in practice; in
particular, when we draw charts, the group $\Ext^{s,f,w}$ will be 
located at Cartesian coordinates $(s,f)$.

The motivic Adams spectral sequence takes the form
\[ E_2=\Ext^{s,f,w}_{\cA_*}(\M_2,\M_2) \Rightarrow
\hat\pi_{s,w},
\]
with $d_r\colon \Ext^{s,f,w}\ra \Ext^{s-1,f+r,w}$.  Here
$\hat\pi_{*,*}$ is the stable motivic homotopy ring of the 
completion of the motivic sphere spectrum with respect to 
the motivic Eilenberg-Mac Lane spectrum $H \F_2$.
(According to \cite{HKO11a},
this completion is also the $2$-completion of the motivic
sphere spectrum, but this is not essential for our calculations.)

Our methods also require us to consider the motivic cohomology
of $\C$ and the motivic Steenrod algebra over $\C$.
We write $\M_2^\C$ and $\sA^\C$ for these objects.
They are obtained from $\M_2$ and $\sA$ by setting $\rho$ equal to zero.
More explicitly, $\M_2^\C$ equals $\F_2[\tau]$,
and the dual motivic Steenrod alegbra over $\C$ has relations of the
form $\tau_k^2 = \tau \xi_{k+1}$.

We will also
abbreviate
\[ \Ext_\R=\Ext_{\cA_*}(\M_2,\M_2), \qquad
\Ext_\C=\Ext_{\cA_*^\C}(\M_2^\C,\M_2^\C).
\]

\subsection{Milnor-Witt degree}
Given a class with an associated stem $s$ and weight $w$, we call $s-w$ the
\dfn{Milnor-Witt degree} of the class.  The terminology comes from the
fact that the elements of Milnor-Witt degree zero in the motivic
stable homotopy ring constitute Morel's Milnor-Witt
$K$-theory ring.  More generally, the elements of Milnor-Witt degree
$r$ in $\hat\pi_{*,*}$ form a module over (2-completed) Milnor-Witt
$K$-theory.

Many of the calculations in this paper are handled by breaking things
up into the homogeneous Milnor-Witt components.  The following lemma
about $\Ext_\C$ will be particularly useful.

\begin{lemma}
\label{lem:MW-bound}
Let $x$ be a non-zero class in $\Ext^{s,f,w}_{\C}$ with
Milnor-Witt degree $t$.
Then $f \geq s - 2t$.
\end{lemma}

\begin{proof}
The motivic May spectral sequence \cite{DI10} has $E_1$-page generated by classes
$h_{ij}$, and converges to $\Ext^\C$.  
All of the classes $h_{ij}$ are readily checked to satisfy the
inequality $s+f-2w\geq 0$, and this extends to all products. 

This inequality is the same as 
$f \geq s - 2(s-w)$, and $t$ equals $s-w$ by definition.
\end{proof}

In practice, Lemma \ref{lem:MW-bound} tells us where to look for elements
of $\Ext_\C$ in a given Milnor-Witt degree.  All such elements lie
above a line of slope 1 on an Adams chart.


\section{The $\rho$-Bockstein spectral sequence}
\label{sctn:p-Bockstein:generalities}

Our aim is to compute $\Ext_{\cA}(\M_2,\M_2)$.  What makes this
calculation difficult is the presence of $\rho$.  If one formally sets
$\rho=0$ then the formulas become simpler and the calculations more
manageable; this is essentially the case that was handled in
\cite{DI10} and \cite{Isaksen14c}.  
Following ideas of Hill \cite{Hill11}, we use an algebraic
spectral sequence for building up the general calculation from the
simpler one where $\rho=0$.  This section sets up the spectral
sequence and establishes some basic properties.

\vspace{0.2in}

Let $\cobar$ be the (unreduced) cobar complex for the Hopf algebroid
$(\M_2,\cA_*)$.  Recall that this is the cochain complex associated to
the cosimplicial ring
\[ \xymatrix{
\M_2\ar@<0.5ex>[r]^{\eta_R}\ar@<-0.5ex>[r]_{\eta_L}
& \cA_* \ar@<0.7ex>[r] \ar[r] \ar@<-0.7ex>[r]   
& \cA_* \tens_{\M_2} \cA_*
\ar@<0.8ex>[r]
\ar@<-0.8ex>[r]
\ar@<0.3ex>[r]
\ar@<-0.3ex>[r]
& \cA_* \tens_{\M_2} \cA_* \tens_{\M_2} \cA_*
\cdots 
}
\]
by taking $d_\cobar$ to be the alternating sum of the coface maps.  
For $u$ an  $r$-fold tensor, one has $d^0(u)=1\tens u$, $d^{r+1}(u)=u\tens
1$, and $d^i(u)$ applies the diagonal of $\cA_*$ to the $i$th tensor
factor of $u$. For $u$ in $\M_2$ (i.e., a $0$-fold tensor),
one has $d^0(u)=\eta_R(u)$ and $d^1(u)=\eta_L(u)$.  

The pair
$(\cobar,d_\cobar)$ is a differential graded algebra.  As usual, we will denote
$r$-fold tensors via the bar notation $[x_1|x_2|\cdots|x_r]$.  

The element $\xi_1^{2^k}$ is primitive in $\sA_*$ for any $k$ because
$\xi_1$ is primitive.
Hence $[\xi_1^{2^k}]$ is a cycle in the cobar complex that is 
denoted $h_{k+1}$.
Likewise, $\tau_0$ is primitive, and
the cycle $[\tau_0]$ is denoted $h_0$.   

The maps $\eta_L$, $\eta_R$, and $\Delta$ all fix $\rho$, and this
implies that all the coface maps are $\rho$-linear.  The filtration
\[ \cobar \supseteq \rho\cobar \supseteq \rho^2\cobar \supseteq \cdots
\]
is therefore a filtration of chain complexes.  The associated spectral
sequence is called the \mdfn{$\rho$-Bockstein spectral sequence}. 

The $\rho$-Bockstein spectral sequence has the form
\[ E_1=\Ext_{\Gr_\rho \cA}(\Gr_\rho \M_2,\Gr_\rho \M_2)
\Rightarrow \Ext_{\cA}(\M_2,\M_2),
\]
where $\Gr_\rho$ refers to the associated graded of the filtration
by powers of $\rho$.
Since $\M_2=\F_2[\tau,\rho]$, we have $\Gr_\rho \M_2\iso \M_2$.
Similarly, it follows easily that there is an isomorphism of Hopf
algebroids
\[ (\Gr_\rho \M_2,\Gr_\rho \cA)\iso (\M_2^\C,\cA^{\C})\tens_{\F_2}
\F_2[\rho],
\]
where $\M_2^\C = \F_2[\tau]$ is the motivic cohomology ring of $\C$.
The point here is that after taking associated gradeds, the formulas
for $\eta_L$ and $\eta_R$ both fix $\tau$, whereas the formulas for
$\Delta$ are unchanged; and all of this  exactly matches the formulas for
$\sA^\C$.   Tensoring with $\F_2[\rho]$ commutes with $\Ext$, 
and so our $\rho$-Bockstein spectral sequence takes the form
\[ E_1=\Ext_{\cA^\C}(\M_2^\C,\M_2^\C) [\rho]
\Rightarrow \Ext_{\cA}(\M_2,\M_2).
\]
It will be convenient to denote
$\Ext_{\cA^\C}(\M_2^\C,\M_2^\C)$ simply by
$\Ext_\C$.  
 
We observe two general properties of the $\rho$-Bockstein spectral sequence.
First, the element $\rho$ is a permanent cycle because $\rho$ supports no
Steenrod operations.
Second, the spectral sequence is multiplicative, so the Leibniz rule
can be used effectively to compute differentials on decomposable elements.

\begin{remark}
\label{rem:rho-ss}
Here is a method for deducing $\rho$-Bockstein differentials
from explicit cobar calculations.
Let $u$ be an element in $\cobar$, and assume that
$u$ is not a multiple of $\rho$.
If possible, write
$d_\cobar(u) = \rho d_\cobar(u_1) + \rho^2 v_2$, where 
$u_1$ has a tensor expression that does not involve $\rho$;
then the $\rho$-Bockstein differential $d_1(u)$ is zero.
Otherwise,  $d_1(u)$ equals $d_\cobar(u)$ modulo $\rho^2$.

If $d_1(u)$ is zero, then if possible write
$d_\cobar(u) = \rho d_\cobar(u_1) + \rho^2 d_\cobar(u_2) + \rho^3 v_3$,
where $u_2$ has a tensor expression that does not involve $\rho$;
then  $d_2(u)$ is zero.
Otherwise, $d_2(u)$ equals $\rho^2 v_2$ modulo $\rho^3$.

Inductively, assume that 
\[
d_\cobar(u) = \rho d_\cobar(u_1) + \cdots + \rho^{r-1} d_\cobar(u_{r-1} ) +
\rho^r v_r,
\]
where
each $u_i$ has a tensor expression that does not involve $\rho$.
If possible, write $v_r = d_\cobar(u_r) + \rho v_{r+1}$,
where
$u_r$ has a tensor expression that does not involve $\rho$;
then $d_r(u)$ is zero.
Otherwise, $d_r(u)$ equals $\rho^r v_r$ modulo $\rho^{r+1}$.
\end{remark}

The method described in 
Remark~\ref{rem:rho-ss} is mostly not needed in our analysis; in
fact, we will eventually show how to deduce a most of the differentials in
a large range of the spectral sequence by a completely mechanical process.
Still, it
is often useful to understand that the $\rho$-Bockstein spectral
sequence is all about computing $\rho$-truncations of differentials in
$\cobar$.  Proposition \ref{prop:diffs-cobar} and Example \ref{ex:d}
illustrate this technique.

\begin{prop}
\label{prop:diffs-cobar}
\mbox{}
\begin{enumerate}
\item
$d_1(\tau)=\rho h_0$.
\item
$d_{2^k}(\tau^{2^k})=\rho^{2^{k}}\tau^{2^{k-1}} h_{k}$ for $k \geq 1$.
\end{enumerate}
\end{prop}

Part (2) of Proposition \ref{prop:diffs-cobar} implicitly also means
that $d_r (\tau^{2^k})$ is zero for all $r < 2^k$.

\begin{proof}
Note that $d_\cobar(x)=\eta_R(x)-\eta_L(x)$ for $x$ in $\M_2$.  In
particular,
$d_\cobar(\tau)=[\tau+\rho \tau_0] - [\tau] = \rho[\tau_0]=\rho h_0$.
Now use Remark~\ref{rem:rho-ss} to deduce that $d_1(\tau)=\rho h_0$.

Next we analyze $d_\cobar(\tau^{2^k})$.  Start with
\[
d_{\cobar}(\tau^{2^k})=
\eta_R(\tau^{2^k})-\eta_L(\tau^{2^k})
= [(\tau +\rho \tau_0 )^{2^k}] - [\tau^{2^k}] =
\rho^{2^k}[\tau_0^{2^k}].
\]
Recall that $\tau_0^2=\tau\xi_1 + \rho\tau_1 + \rho\tau_0\xi_1$ in $\cA_*$, and
so
$\tau_0^{2^k}=\tau^{2^{k-1}}\xi_1^{2^{k-1}}$ modulo $\rho^{2^{k-1}}$.
Thus,
$d_{\cobar}(\tau^{2^k})=\rho^{2^k}\tau^{2^{k-1}}[\xi_1^{2^{k-1}}]$
modulo $\rho^{2^k+1}$.
Remark~\ref{rem:rho-ss} implies that
$d_{2^k}(\tau^{2^k}) = \rho^{2^k} \tau^{2^{k-1}} h_k$.
\end{proof}

\begin{example}
\label{ex:d}
We will demonstrate that $d_6(\tau^4 h_1)=\rho^6\tau h_2^2$.
As in the proof of Proposition \ref{prop:diffs-cobar},
$d_\cobar(\tau^4) = \rho^4[(\tau \xi_1 +\rho \tau_1 + \rho \tau_0\xi_1)^2]$.
Use the relations
$\tau_0^2 = \tau \xi_1 + \rho \tau_1 + \rho \tau_0 \xi_1$ and
$\tau_1^2 = \tau \xi_2 + \rho \tau_2 + \rho \tau_0 \xi_2$ to see that this expression equals
$\rho^4 \tau^2 [\xi_1^2] + \rho^6 \tau [\xi_2] + \rho^6 \tau [\xi_1^3]$
modulo $\rho^7$.

Since $h_1=[\xi_1]$ is a cycle, we therefore have
\[ d_\cobar(\tau^4h_1)=\rho^4\tau^2[\xi_1^2|\xi_1] + \rho^6\tau\bigl
([\xi_2|\xi_1]+[\xi_1^3|\xi_1]\bigr )
\]
modulo $\rho^7$.

The coproduct on $\xi_2$ implies that
$d_\cobar([\xi_2]) = [\xi_1^2| \xi_1]$.
We also have that
\[
d_\cobar(\tau^2) = \rho^2 [\tau_0^2] = 
\rho^2 \tau [\xi_1] + \rho^3 [\tau_1] + \rho^2 [\tau_0 \xi_1],
\]
as in the proof of Proposition \ref{prop:diffs-cobar}.
Therefore, the Leibniz rule gives that
\[
d_\cobar(\tau^2 [\xi_2]) = \rho^2\tau
[\xi_1|\xi_2]+\rho^3[\tau_1|\xi_2]+\rho^3[\tau_0\xi_1|\xi_2] + \tau^2[\xi_1^2|\xi_1].
\]

We can now write
\[ d_\cobar(\tau^4h_1)=\rho^4d_\cobar(\tau^2[\xi_2])+\rho^6\tau\bigl (
[\xi_2|\xi_1]+[\xi_1^3|\xi_1]+[\xi_1|\xi_2]
\bigr)
\]
modulo $\rho^7$.
From Remark \ref{rem:rho-ss},
one has $d_i(\tau^4 h_1)=0$ for $i<6$
in the $\rho$-Bockstein spectral sequence, and
$d_6(\tau^4h_1)=\rho^6\tau\bigl (
[\xi_2|\xi_1]+[\xi_1^3|\xi_1]+[\xi_1|\xi_2]\bigr )$. 

Finally, the coproduct in $\cA_*$ implies that 
\[ d_\cobar([\xi_2\xi_1])=[\xi_1^3|\xi_1]+[\xi_1|\xi_2]
+[\xi_2|\xi_1]+[\xi_1^2|\xi_1^2].
\]
This shows that $[\xi_2|\xi_1]+[\xi_1^3|\xi_1]+[\xi_1|\xi_2]=h_2^2$ in
$\Ext$.  
\end{example}

The long analysis in Example~\ref{ex:d} 
demonstrates that direct work with the cobar complex is not practical.
Instead, we will use some clever tricks that take
advantage of various algebraic structures.  But it is useful to
remember what is going on behind the scenes: these computations of
differentials are always giving us clues about the cobar differential
$d_\cobar$.

The following two results are useful in analyzing $\rho$-Bockstein
differentials.

\begin{lemma}
\label{lem:differential-torsion}
If $d_r(x)$ is non-trivial in the $\rho$-Bockstein spectral sequence, 
then $x$ and $d_r(x)$ are both $\rho$-torsion free on the $E_r$-page.
\end{lemma}

\begin{proof}
First note that if $y$ is nonzero on the $E_r$-page, then $y$ is
$\rho$-torsion if and only if $\rho^{r-1}y=0$.  The reason is that the
differentials $d_s$ for $s<r$ can only hit $\rho^s$-multiples of $y$.

Now suppose that $d_r(x) = \rho^r y$, where $\rho^r y$ is non-zero on the $E_r$-page.
This immediately forces $y$ to be $\rho$-torsion free.
Since $d_r$ is $\rho$-linear, this implies that
$x$ must also be $\rho$-torsion free on the $E_r$-page.
\end{proof}

%
%

\section{$\rho$-localization}
\label{sctn:p-local}

The analysis of the $\rho$-Bockstein spectral sequence is best broken
up into two pieces.
There are a large number of $\rho$-torsion
classes in the $E_\infty$-page.
If one throws away all of this $\rho$-torsion, then
the end result turns out to be fairly simple.  In this section we
compute this simple piece of $\Ext_\R$.  
More precisely, we will consider the $\rho$-localization
$\Ext_\R[\rho^{-1}]$ of $\Ext^\R$.
Inverting $\rho$ annihilates all of the $\rho$-torsion.

\vspace{0.2in}

Let $\cA^{\cl}$ denote the classical Steenrod algebra (at the prime
$2$), and write $\Ext_{\cl}=\Ext_{\cA^{\cl}}(\F_2,\F_2)$.  

\begin{thm}
\label{thm:p-local}
There is an isomorphism from
$\Ext_{\cl} [\rho^{\pm 1}]$ to 
$\Ext_\R[\rho^{-1}]$ such that:
\begin{enumerate}
\item
The isomorphism is highly structured, i.e., preserves products, Massey products,
and algebraic squaring operations in the sense of \cite{May70}.
\item
The element $h_n$ of $\Ext_{\cl}$ corresponds to the element
$h_{n+1}$ of $\Ext_\R$.
\item
An element in $\Ext_{\cl}$ of degree $(s,f)$ corresponds to
an element in $\Ext_\R$ of degree $(2s+f, f, s+f)$.
\end{enumerate}
\end{thm}

The formula for degrees appears to be more complicated than it is.
The idea is that one doubles the internal degree, 
which is the stem plus the Adams filtration, while leaving the
Adams filtration unchanged.
Then the weight
is always exactly half of the internal degree.

\begin{proof}
Since localization is exact, we may compute the cohomology
of the Hopf algebroid $(\M_2[\rho^{-1}], \cA_*[ \rho^{-1} ])$ to obtain 
$\Ext^\R[\rho^{-1}]$.
After localizing at $\rho$, we have
$\tau_{k+1} = \rho^{-1} \tau_k^2 + 
  \rho^{-1} \tau \xi_{k+1} + \tau_0 \xi_{k+1}$, 
and so 
the Hopf algebroid $(\M_2[\rho^{-1}], \cA_*[ \rho^{-1} ])$ is described by
\begin{align*}
& \cA_*[\rho^{-1}] = \M_2[\rho^{-1}] [ \tau_0, \xi_0, \xi_1, \ldots ] /
  (\xi_0 = 1) \\
& \eta_L(\tau) = \tau, \quad 
\eta_R(\tau) = \tau + \rho \tau_0, \quad \eta_L(\rho)=\eta_R(\rho)=\rho \\
& \Delta( \tau_0 ) = \tau_0 \otimes 1 + 1 \otimes \tau_0 \\
& \Delta( \xi_k ) = \sum \xi_{k-i}^{2^i} \otimes \xi_i.
\end{align*}
Since these formulas contain no interactions between $\tau_i$'s and
$\xi_j$'s, there is a splitting
\[
(\M_2[\rho^{-1}], \cA_*[\rho^{-1}])\iso 
(\M_2[\rho^{-1}],\cA'_*) \otimes_{\F_2} (\F_2, \cA''_*),
\]
where
$(\M_2[\rho^{-1}],\cA'_*)$ is the Hopf algebroid
\begin{align*}
& \cA'_* = \M_2[\rho^{-1}] [ \tau_0] \\
& \eta_L(\tau) = \tau, \quad \eta_R(\tau)= \tau + \rho \tau_0 \\
& \Delta( \tau_0 ) = \tau_0 \otimes 1 + 1 \otimes \tau_0
\end{align*}
and $(\F_2,\cA''_*)$ is the Hopf algebra
\begin{align*}
& \cA''_* = \F_2 [ \xi_0, \xi_1, \ldots ] /
  (\xi_0 = 1) \\
& \Delta( \xi_k ) = \sum \xi_{k-i}^{2^i} \otimes \xi_i.
\end{align*}
Notice that $\cA''_*$ is equal to the classical dual Steenrod algebra, 
and so its cohomology is $\Ext_{\cl}$ (with
degrees suitably shifted).  For $\cA'_*$, 
we can perform the change of variables $x=\rho\tau_0$
since $\rho$ is invertible, 
yielding 
\[
(\M_2[\rho^{-1}], \cA'_*) \iso \F_2[\rho^{\pm 1}] \tens_{\F_2} (\F_2[\tau],\cB), 
\]
where $(\F_2[\tau], \cB)$ is
the Hopf algebroid defined in Lemma~\ref{lem:A'-lemma} below.
The lemma implies that the cohomology of $(\M_2[\rho^{-1}],\cA'_*)$ 
is $\F_2[\rho^{\pm 1}]$, concentrated in homological degree zero.
\end{proof}

\begin{lemma}
\label{lem:A'-lemma}
Let $R=\F_2[t]$ and let $\cB=R[x]$, with Hopf algebroid structure on
$(R,\cB)$ given by the formulas
\[ \eta_L(t)=t, \quad \eta_R(t)=t+x, \quad \Delta(x)=x\tens 1+
1\tens x
\]
(the formula $\Delta(t)=t\tens 1$ is forced by the axioms).  
Then the cohomology of $(R,\cB)$ is isomorphic to $\F_2$, concentrated
in homological degree $0$.  
\end{lemma}

\begin{proof}
Let $\cobar_\cB$ be the cobar complex of $(R,\cB)$, and 
filter by powers of $x$.  More explicitly, let $F_i \cobar_\cB$ be the subcomplex
\[ 
0 \ra x^i \cB \ra \sum_{p+q=i} x^p \cB\tens_R x^q \cB \ra \sum_{p+q+r=i}
x^p \cB \tens_R x^q \cB\tens_R x^r \cB \ra \cdots 
\]
This is indeed a subcomplex, and the associated graded $\Gr_x \cobar_\cB$ 
is the cobar complex for $(R,\Gr_x \cB)$.  This pair is
isomorphic to the Hopf algebra (no longer a Hopf algebroid) where
$\eta_L(t)=\eta_R(t)=t$ and $\Delta(x)=x\tens 1+1\tens x$.  The
associated cohomology is the infinite polynomial algebra
$\F_2[t,h_0,h_1,h_2,\ldots]$, where $h_i=[x^{2^i}]$.  One easy way to
see this is to note that the dual of $\Gr_x \cB$ is the exterior algebra
$\F_2[t](e_0,e_1,e_2,\ldots)$, where $e_i$ is dual to $x^{2^i}$.  

Our filtered cobar complex gives rise to a multiplicative spectral sequence with
$E_1$-page equal to $\F_2[t,h_0,h_1,\ldots]$ and converging to the cohomology
of $(R,\cB)$.  The classes $h_i$ are all infinite cycles, since
$[x^{2^i}]$ is indeed a cocycle in $\cobar_\cB$.  Essentially the same analysis as in
Proposition~\ref{prop:diffs-cobar} shows that $d_1(t)=h_0$. This shows that
the $E_2$-page is $\F_2[t^2,h_1,h_2,\ldots]$.
The analysis from Proposition~\ref{prop:diffs-cobar} again shows
$d_2(t^2)=h_1$, which implies that the $E_3$-page is
$\F_2[t^4,h_2,h_3,\ldots]$.  Continue inductively, using that
$d_{2^i}(t^{2^i})=h_i$.  The $E_\infty$-page is just $\F_2$.
\end{proof}

\begin{remark}
We gave a calculational proof of Lemma \ref{lem:A'-lemma}.
Here is a sketch of a more conceptual proof.

The Hopf algebroid $(R,\cB)$ has
the same information as the presheaf of groupoids which sends an
$\F_2$-algebra $S$ to the groupoid with object set
$\Hom_{\F_2\!-\!\alg}(R,S)$ and morphism set $\Hom_{\F_2\!-\!\alg}(\cB,S)$.  One
readily checks that this groupoid is the translation category
associated to the abelian group $(S,+)$; very briefly, the image of
$x$ in $S$ is the name of the morphism, the image of $t$ is its
domain, and therefore $t+x$ is its codomain.  Notice
that this groupoid
is contractible no matter what $S$ is---this is the key
observation. By 
\cite{hovey}*{Theorems A and B}
it follows that the category of $(R,\cB)$-comodules is equivalent
to the category of comodules for the trivial Hopf algebroid
$(\F_2,\F_2)$.  In particular, one obtains an isomorphism of $\Ext$ groups.
\end{remark}


\section{Analysis of the $\rho$-Bockstein spectral sequence}
\label{sctn:Bockstein-analysis}

In this section we determine all differentials in the $\rho$-Bockstein
spectral sequence, within a given range of dimensions.

\vspace{0.1in}
 
\subsection{Identification of the $E_1$-page}

From Section \ref{sctn:p-Bockstein:generalities}, 
the $\rho$-Bockstein spectral sequence takes the form
\[ E_1=\Ext_{\C} [\rho] \Rightarrow \Ext_{\cA}(\M_2,\M_2).
\]
The groups $\Ext_\C$ have been
computed in \cite{DI10} and \cite{Isaksen14c} through a large range of
dimensions. 
Figure \ref{fig:ExtC} gives a picture of
$\Ext_\C$.  Recall that this chart is a two-dimensional
representation of a tri-graded object.  For every black dot  $x$ in
the chart there are classes $\tau^i x$ for $i \geq 1$ 
lying behind $x$ (going into the page); in contrast, the red dots are
killed by $\tau$.  To get the $E_1$-page for the $\rho$-Bockstein
spectral sequence, we freely adjoin the class $\rho$ to this chart.
With respect to the picture, multiplication by $\rho$ moves one degree
to the left
and one degree back.  So we can regard the same chart as a depiction
of our $E_1$-page if we interpret every black dot as representing an
entire triangular cone moving back (via multiplication by $\tau$) and to
the left (via multiplication by $\rho$); and every red dot represents
a line of $\rho$-multiples going back and to the left.  For example,
we must remember that in the $(2,1)$ spot on the grid there are
classes $\rho\tau^i h_2$, $\rho^5\tau^i h_3$, $\rho^{13}\tau^i h_4$,
and so forth.  In general, when looking at coordinates $(s,f)$ on the chart,
one must look horizontally to the right and be aware that $\rho^k x$
is potentially present, where $x$ is a class in $\Ext_\C$ at 
coordinates $(s+k,f)$.

There are so many classes in the $E_1$-page, and it is so difficult to
represent the three-dimensionsal chart, that one of the largest challenges of
running the $\rho$-Bockstein spectral sequence is one of
organization.  We will explain some techniques for managing this.

\subsection{Sorting the $E_1$-page}

To analyze the $\rho$-Bockstein spectral sequence it is useful to sort
the $E_1$-page by the Milnor-Witt degree $s-w$.  
The $\rho$-Bockstein differentials all have degree $(-1,1,0)$ with
respect to the $(s,f,w)$-grading, and therefore have degree $-1$ with
respect to the Milnor-Witt degree.

Table \ref{tab:MW-gens} shows the multiplicative generators for 
the $\rho$-Bockstein $E_1$-page
through Milnor-Witt degree $5$.
The information in Table \ref{tab:MW-gens} was extracted from
the $\Ext_\C$ chart in Figure \ref{fig:ExtC} in the following manner.
Lemma \ref{lem:MW-bound} says that elements
in Milnor-Witt degree $t$ satisfy
$f\geq s-2t$.  
Specifically, elements in Milnor-Witt degree at most $5$ lie
on or above the line $f = s - 10$ of slope $1$.

This region is infinite, and in principle could contain
generators in very high stems.  However, in $\Ext_\C$ there is a 
line of slope $1/2$ above which all elements are multiples of $h_1$
\cite{GI15}.  The line of slope $1$ 
and the line of slope $1/2$ bound a finite region which is 
easily searched exhaustively for generators of
Milnor-Witt degree at most $5$.

Note that the converse does not hold:
some elements bounded by these lines may have
Milnor-Witt degree greater than $5$.

Our $E_1$-page is additively generated by all nonvanishing
products of the elements from Table~\ref{tab:MW-gens}.
 Because the Bockstein differentials are $\rho$-linear, it
suffices to understand how the differentials behave on products that
do not involve $\rho$.    
Table \ref{tab:rho-Bock-gens}
shows $\F_2[\rho]$-module generators for the $E_1$-page,
sorted by Milnor-Witt degree.

\subsection{Bockstein differentials}
\mbox{}

Proposition \ref{prop:diffs-cobar} established some 
$\rho$-Bockstein differentials with a brute force approach via the 
cobar complex. 
We will next 
describe a different
technique that computes all differentials in a large range.

All of our arguments will center on the $\rho$-local calculation
of Theorem~\ref{thm:p-local}.  This result says that if we invert
$\rho$, then  the $\rho$-Bockstein spectral sequence converges to a
copy of $\Ext_{\cl}\tens_{\F_2} \F_2[\rho,\rho^{-1}]$, with the motivic
$h_i$ corresponding to the classical $h_{i-1}$. 

When identifying possible $\rho$-Bockstein differentials,
there are two useful things to keep in mind:
\begin{itemize}
\item With respect to our $\Ext_\C$ chart, the differentials all go up
one spot and left one spot;
\item With respect to Table~\ref{tab:rho-Bock-gens}, the differentials
all go to the left one column.  
\end{itemize}
Combining these two facts (which
involves switching back and forth between the chart and table), one
can often severely narrow the possibilities for differentials.

\begin{lemma}
\label{lem:d1}
\mbox{}
The $\rho$-Bockstein $d_1$ differential is completely determined by:
\begin{enumerate}
\item
$d_1(\tau) = \rho h_0$.
\item
The elements $h_0$, $h_1$, $h_2$, $h_3$, and $c_0$ are all permanent
cycles.
\item $d_1(Ph_1)=0$.  
\end{enumerate}
\end{lemma}

\begin{proof}
The differential $d_1(\tau) = \rho h_0$ was established in
Proposition \ref{prop:diffs-cobar}.

The classes $h_0$ and $h_1$ cannot support differentials because
there are no elements in negative Milnor-Witt degrees.
The classes $h_2$ and $h_3$ must survive the $\rho$-local spectral
sequence, so they cannot support differentials.  Comparing chart and
table, there are no possibilities for a differential on $c_0$.

Finally, if $d_1(Ph_1)$ is nonzero, then it is of the form $\rho x$ for
a class $x$ that does not contain $\rho$. 
This class $x$ would appear at coordinates $(9,6)$ 
in the $\Ext_\C$ chart.  By inspection, there is no such $x$.
\end{proof}

\begin{remark}
We have shown that $P h_1$ survives to the $E_2$-page,
but we have not shown that it is a permanent cycle.
The $\Ext_\C$
chart shows that $\rho^3 h_1^3 c_0$ is the only potential target
for a differential on $P h_1$.
If $P h_1$ is not a permanent cycle, then 
the only possibility
is that $d_3(Ph_1)$ equals $\rho^3 h_1^3 c_0$.
We will see below in Lemma \ref{lem:d3} that this
differential does occur.  
\end{remark}

Lemma~\ref{lem:d1} allows us to compute all $d_1$-differentials, using
the product structure. 
Figure \ref{fig:Bockstein} displays the resulting $E_2$-page,
sorted by Milnor-Witt degree.

Table \ref{tab:rho-Bock-E2}
gives $\F_2[\rho]$-module generators
for part of the $E_2$-page. 
Recall from Lemma \ref{lem:differential-torsion} that 
$\rho$-torsion elements cannot be involved in any further differentials,
so we have not included such elements in the table.
We have also eliminated the elements that cannot be
involved in any differentials because 
we know they are $\rho$-local by Theorem \ref{thm:p-local}.

Note that $\tau h_1$ is indecomposable in the $E_2$-page, although
$\tau h_1^2$ does decompose as $\tau h_1 \cdot h_1$.  The
multiplicative generators for the $E_2$-page are then
\[ \boxed{h_0},\  \boxed{h_1},\  \tau h_1,\  \boxed{h_2}, \ \tau^2,\  \tau h_2^2,\  \boxed{h_3},\
\boxed{c_0},\  \tau
h_0^3h_3,\  \tau c_0,\  Ph_1,
\]
where boxes indicate classes that we already know are permanent cycles.

\begin{lemma}
\label{lem:d2}
The $\rho$-Bockstein $d_2$ differential is completely determined by:
\begin{enumerate}
\item
$d_2(\tau^2) = \rho^2 \cdot \tau h_1$.
\item
The elements $\tau h_1$, $\tau h_2^2$, and $\tau c_0$ are permanent cycles.
\item
$d_2(\tau h_0^3 h_3) = 0$.
\item
$d_2(P h_1) = 0$.
\end{enumerate}
\end{lemma}

\begin{proof}
The differential $d_2(\tau^2) = \rho^2 \tau h_1$ was established in 
Proposition \ref{prop:diffs-cobar}.

Comparison of chart and table shows that
a Bockstein differential on $\tau h_1$ could only hit $h_0^2$ or
$\rho^2 h_1^2$.  The first is impossible since the target of a 
$d_2$ differential must be divisible by $\rho^2$, and the second is ruled
out by the fact that $h_1^2$ survives
$\rho$-localization.  So no differential can ever exist on $\tau h_1$.  

Similarly, chart and table show that there are no possible
differentials on $\tau h_2^2$, and 
no possible
$d_2$ differential on either $\tau h_0^3 h_3$ or $Ph_1$.  

It remains to consider $\tau c_0$.  The only possibility for a
differential is that $d_2 (\tau c_0)$ might equal $\rho^2 h_1 c_0$.
But if this happened we would also have $d_2(h_1^2\tau c_0)= \rho^2
h_1^3 c_0$, which contradicts the fact that 
$\tau h_1^2 c_0$ is zero on the $E_2$-page, while
$\rho^2 h_1^3  c_0$ is non-zero.
\end{proof}

Once again, Lemma \ref{lem:d2} allows the complete computation of the
$E_3$-page (in our given range), which is
shown in Figure \ref{fig:Bockstein}, sorted by Milnor-Witt degree.
Table \ref{tab:rho-Bock-E3}
gives $\F_2[\rho]$-module generators
for part of the $E_3$-page.
Recall from Lemma \ref{lem:differential-torsion} that 
$\rho$-torsion elements cannot be involved in any further differentials,
so we have not included such elements in the table.
We have also eliminated the elements that cannot be
involved in any differentials because
we know they are $\rho$-local by Theorem \ref{thm:p-local}.

The multiplicative generators for the $E_3$-page are
\[ \boxed{h_0},\  \boxed{h_1},\  \boxed{\tau h_1},\  \boxed{h_2},\  \tau^2
h_0,\  \tau^2 h_2,\  \boxed{\tau h_2^2},\  \boxed{h_3}, \
\boxed{c_0},\  \tau^4,\ 
\tau h_0^3 h_3, \ \boxed{\tau c_0},\  Ph_1,
\]
where boxes indicate classes that we already know are permanent cycles.

\begin{lemma} The $\rho$-Bockstein $d_3$ differential is 
completely determined by:
\label{lem:d3}
\mbox{}
\begin{enumerate}
\item
$d_3(P h_1) = \rho^3 h_1^3 c_0$.
\item
The elements $\tau^2 h_0$ and $\tau^2 h_2$ are permanent cycles.
\item
$d_3( \tau^4) = 0$.
\item
$d_3(\tau h_0^3 h_3) = 0$.
\end{enumerate}
\end{lemma}

\begin{proof}
As we saw in Lemma~\ref{lem:d1}, 
$h_1$ and  $c_0$ are permanent cycles. 
Therefore, $h_1^3 c_0$ is a permanent cycle.
We know from Theorem \ref{thm:p-local} that $h_1^3 c_0$
does not survive $\rho$-localization.  Therefore,
some differential hits $\rho^r h_1^3 c_0$.
The only possibility is that $d_3(P h_1)$ equals $\rho^3 h_1^3 c_0$.

Inspection of the $E_3$-page shows that
there are no possible values for differentials on $\tau^2 h_0$.
For $\tau^2 h_2$, there is a possibility that
$d_4(\tau^2 h_2)$ equals $\rho^2 h_2^2$.
However, this differential is ruled out by Theorem \ref{thm:p-local}.

By inspection, there are no possible values for $d_3$ differentials
on $\tau^4$ or $\tau h_0^3 h_3$.
\end{proof}

The $d_3$ differential has a very mild effect on the $E_3$-page of our
spectral sequence.  In Table~\ref{tab:rho-Bock-E3}, the elements
$Ph_1$ and $h_1^kPh_1$ disappear from column four, 
the elements $h_1^k c_0$ disappear from column three for $k \geq 3$.
Everything else remains the same, so we will not include a separate
table for the $E_4$-page.  
The multiplicative generators are
the same as for the $E_3$-page, except that $Ph_1$ is thrown out.
Figure \ref{fig:Bockstein} depicts the $E_4$-page, sorted by
Milnor-Witt degree.

Also, all of these generators are permanent cycles 
except possibly for $\tau^4$ and
$\tau h_0^3h_3$.  
In particular, 
every element of the $E_4$-page in Milnor-Witt degrees strictly less
than $4$ is now known to be a permanent cycle.
All the remaining differentials will go from 
Milnor-Witt degree $4$ to Milnor-Witt degree $3$.

\begin{lemma}
\label{lem:d4}
The $\rho$-Bockstein $d_4$ differential is completely determined by:
\begin{enumerate}
\item
$d_4(\tau^4) = \rho^4 \tau^2 h_2$.
\item
$d_4(\tau h_0^3 h_3) = \rho^4 h_1^2 c_0$.
\item
The other generators of the $E_4$-page are permanent cycles.
\end{enumerate}
\end{lemma}

\begin{proof}
The differential
$d_4(\tau^4) = \rho^4 \tau^2 h_2$ was established in
Proposition \ref{prop:diffs-cobar}.

We know that $h_1^2c_0$ is a permanent cycle, but we also 
 know from Theorem \ref{thm:p-local} that $h_1^2 c_0$
does not survive $\rho$-localization.  Therefore,
some differential hits $\rho^r h_1^2 c_0$ for some $r$.  Looking at
the chart,
the only possibility is that $d_4(\tau h_0^3 h_3)$ equals $\rho^4 h_1^2 c_0$.
\end{proof}

The multiplicative generators of the $E_5$-page are the permanent
cycles we have seen already, together with
$\tau^4 h_0$ and  $\tau^4 h_1$. 

\begin{lemma}
\label{lem:d5}
The $\rho$-Bockstein $d_5$ differential is zero.
\end{lemma}

\begin{proof}
We only have to check for possible $d_5$ differentials on $\tau^4
h_0$ and $\tau^4 h_1$.  Inspection of the $\Ext_\C$ chart
shows that there are no classes in the relevant degrees.
\end{proof}

Figure \ref{fig:Bockstein}
displays the $E_6$-page, sorted by Milnor-Witt degree.

\begin{lemma}
\label{lem:d6}
The $\rho$-Bockstein $d_6$ differential is completely determined by:
\begin{enumerate}
\item
$d_6(\tau^4 h_1) = \rho^6  \tau h_2^2$.
\item
The element $\tau^4 h_0$ is a permanent cycle.
\end{enumerate}
\end{lemma}

\begin{proof}
Lemma \ref{lem:differential-torsion} implies that
$\tau^4 h_0$ is a permanent cycle because of the
differential $d_1(\tau^5) = \rho \tau^4 h_0$.

By Theorem \ref{thm:p-local}, we know that $\tau^4 h_1$ does not survive
$\rho$-localization.  
Since 
$\rho^r \tau^4 h_1$ cannot be hit by a differential, it follows that
$\tau^4 h_1$ supports a differential.
The two possibilities are that 
$d_6(\tau^4 h_1)$ equals $\rho^6 \tau h_2^2$ or
$d_8(\tau^4 h_1)$ equals $\rho^8 h_1 h_3$.
We know from Theorem \ref{thm:p-local} that $h_1 h_3$ survives $\rho$-localization.
Therefore, we must have $d_6(\tau^4 h_1) = \rho^6 \tau h_2^2$.  
\end{proof}

The multiplicative generators for the $E_7$-page are $\tau^4 h_1^2$, 
together with other classes that we already know are permanent cycles.
Figure \ref{fig:Bockstein}
displays the $E_7$-page, sorted by Milnor-Witt degree.
\begin{lemma}
\label{lem:d7}
The $\rho$-Bockstein $d_7$ differential is completely determined by:
\begin{enumerate}
\item
$d_7(\tau^4 h_1^2) = \rho^7 c_0$.
\end{enumerate}
\end{lemma}

\begin{proof}
By Theorem \ref{thm:p-local}, we know that $\tau^4 h_1^2$ does not survive
$\rho$-localization, and  
since 
$\rho^r \tau^4 h_1^2$ cannot be hit by a differential it follows that
$\tau^4 h_1^2$ supports a differential.
The two possibilities are that 
$d_7(\tau^4 h_1^2)$ equals $\rho^7 c_0$ or
$d_8(\tau^4 h_1^2)$ equals $\rho^8 h_1^2 h_3$.
We know from Theorem \ref{thm:p-local} that $h_1^2 h_3$ survives $\rho$-localization.
Therefore, we must have $d_7(\tau^4 h_1^2) = \rho^7 c_0$.
\end{proof}

Finally, once we reach the $E_8$-page, we simply observe that all the
multiplicative generators are classes that have already been checked
to be permanent cycles.

\subsection{The $\rho$-Bockstein $E_\infty$-page}
\label{subsctn:p-Bock-Einfty}

Table \ref{tab:rho-Bock-Einfty}
describes the $\rho$-Bockstein $E_\infty$-page in the
range of interest.  The table gives a list of $\M_2$-module generators
for the $E_\infty$-page.
We write $x \, (\rho^k)$
if $x$ is killed by $\rho^k$, and we write $x$(loc) 
for classes that are non-zero after $\rho$-localization.

The reader is invited to construct a single Adams chart that captures all
of this information.  We have found that combining all of the Milnor-Witt
degrees into one picture makes it too difficult to get a feel
for what is going on.
For example, at coordinates $(3,3)$, one has six elements
$h_1^3$, $\tau h_1^3$, $\tau^2 h_1^3$, $\tau^3 h_1^3$, $\rho^5 c_0$, and
$\rho^6 h_1^2 h_3$.
Each of these elements is related by $h_0$, $h_1$, and $\rho$ extensions
to other elements.


\section{From the $\rho$-Bockstein $E_\infty$-page to $\Ext_\R$}
\label{sctn:ExtR}

Having obtained the $E_\infty$-page of the $\rho$-Bockstein spectral sequence,
we will now compute all hidden extensions in the range under consideration.
The key arguments rely on 
May's Convergence Theorem \cite{May69} in a slightly unusual way.
We will use this theorem to argue that certain Massey products
$\langle a, b, c \rangle$ cannot be well-defined.  We will deduce that
either $a b$ or $b c$ must be non-zero via a hidden extension.

\begin{remark} 
As is typical in this kind of analysis, there are issues underlying
the naming of classes.  An element $x$ of the Bockstein $E_\infty$-page
represents a coset of elements of $\Ext_\R$, and it is
convenient if we can slightly ambiguously use the same symbol $x$ 
for one particular element from this coset.
This selection has to happen on a
case-by-case basis, but once done it  allows us to 
use the same symbols for 
elements of the Bockstein $E_\infty$-page and  for elements
of $\Ext_\R$ that they represent.  

For example,
the element $h_0$ on the $E_\infty$-page represents two elements of $\Ext_\R$,
because of the presence of $\rho h_1$ in higher Bockstein filtration.
One of these elements is annihilated by $\rho$ and the other is not.
We write $h_0$ for the element of $\Ext_\R$ that is annihilated by $\rho$.

Table \ref{tab:Ext-gen} summarizes these ambiguities 
and gives definitions in terms of $\rho$-torsion.
\end{remark}

Once again, careful bookkeeping is critical at this stage.
We begin by choosing preferred $\F_2[\rho]$-module generators
for $\Ext_\R$ up to Milnor-Witt degree $4$.
First, we choose an ordering of the multiplicative generators of
$\Ext_\R$:
\[
\rho < h_0 < h_1 < \tau h_1 < h_2 < \tau^2 h_0 < \tau^2 h_2 <
\tau h_2^2 < h_3 < c_0 < \tau^4 h_0 < \tau c_0.
\]
The ordering here is essentially arbitrary, although it is convenient
to have elements of low Milnor-Witt degree appear first.

Next, we choose $\F_2[\rho]$-module generators for $\Ext_\R$ that come first
in the lexicographic ordering on monomials in these generators.
For example, we could choose either $h_0^2 h_2$ or $\tau h_1 \cdot h_1^2$
to be an $\F_2[\rho]$-module generator; we select $h_0^2 h_2$
because $h_0 < h_1$.
We do this for each element listed in Table \ref{tab:rho-Bock-Einfty}.

The results of these choices are displayed in Table \ref{tab:ExtR-gen}.
This table lists $\F_2[\rho]$-module generators of $\Ext_\R$.
We write $x \, (\rho^k)$
if $x$ is killed by $\rho^k$, and we write $x$(loc) 
for classes that are non-zero after $\rho$-localization.

Our goal is to produce a list of relations for $\Ext_\R$
that allows every monomial to be reduced to a linear combination
of monomials listed in Table \ref{tab:ExtR-gen}.
We will begin by considering all pairwise products of generators.

\begin{lemma}
\label{lem:ExtR-mult}
Through Milnor-Witt degree 4,
Table \ref{tab:ExtR-mult} lists the products of all pairs of
multiplicative generators of $\Ext_\R$.
\end{lemma}

In Table \ref{tab:ExtR-mult}, 
the symbol $-$ indicates that the product has no simpler form,
i.e., is a monomial listed in Table \ref{tab:ExtR-gen}.

\begin{proof}
Some products are zero because there is no other possbility;
for example $h_1 h_2$ is zero because there are no non-zero elements in the
appropriate degree.

Some products are zero because we already know that they are annihilated
by some power of $\rho$, while the only non-zero elements in the 
appropriate degree are all $\rho$-local.  For example,
for degree reasons, it is possible that $h_0 h_1$ equals $\rho h_1^2$.
However, we already know that $\rho h_0$ is zero, while
$h_1^2$ is $\rho$-local.  Therefore, $h_0 h_1$ must be zero.
Similar arguments explain all of the pairwise products that are zero
in Table \ref{tab:ExtR-mult}. 

Some of the non-zero pairwise products are not hidden in the 
$\rho$-Bockstein spectral sequence.
For example, consider the product $\tau^2 h_0 \cdot h_2$.
We have that $\tau^2 h_0 \cdot h_2 + h_0 \cdot \tau^2 h_2$ is zero on the
$\rho$-Bockstein $E_\infty$-page, but
$\tau^2 h_0 \cdot h_2 + h_0 \cdot \tau^2 h_2$ might equal something
of higher $\rho$-filtration in $\Ext_\R$.
The possible values for this expression in $\Ext_\R$ are 
the linear combinations of $\rho^3 \cdot \tau h_2^2$ and
$\rho^5 h_1 h_3$.  Both of these elements are non-zero
after multiplication by $\rho$, while
$\tau^2 h_0 \cdot h_2 + h_0 \cdot \tau^2 h_2$ is annihilated by
$\rho$ in $\Ext_\R$.  Therefore, we must have that 
$\tau^2 h_0 \cdot h_2 + h_0 \cdot \tau^2 h_2 = 0$ 
in $\Ext_\R$.

The same argument applies to the other non-hidden extensions 
in Table \ref{tab:ExtR-mult}, except that they are somewhat easier because
there are no possible hidden values.

The remaining non-zero pairwise products are all hidden in
the $\rho$-Bockstein spectral sequence.
For these, we need a more sophisticated argument involving
Massey products and May's Convergence Theorem \cite{May69}.
This theorem says that under certain technical conditions involving the
vanishing of ``crossing" differentials, one can compute Massey
products in $\Ext_\R$ using the $\rho$-Bockstein differentials.

We will demonstrate how this works for the product $h_0 \cdot \tau h_1$.
Consider the Massey product $\langle \rho, h_0, \tau h_1 \rangle$
in $\Ext_\R$.  If this Massey product were well-defined,
then May's Convergence Theorem and the $\rho$-Bockstein differential
$d_1(\tau) = \rho h_0$ would imply that the Massey product 
contains an element that is detected by
$\tau^2 h_1$ in the $\rho$-Bockstein $E_\infty$-page.
(Beware that one needs to check that there are no crossing differentials.)
The element $\tau^2 h_1$ does not survive to the $E_\infty$-page.
Therefore, the Massey product is not well-defined, so
$h_0 \cdot \tau h_1$ must be non-zero.
The only possible value for the product is $\rho h_1 \cdot \tau h_1$.

The same style of argument works for all of the hidden extensions
listed in Table \ref{tab:ExtR-mult}, with one additional complication in some
cases.
Consider the product $h_1 \cdot \tau^2 h_0$.  
Analysis of the Massey product $\langle \rho, \tau^2 h_0, h_1 \rangle$
implies that the product must be non-zero, since
$\tau^3 h_1$ does not survive to the $\rho$-Bockstein
$E_\infty$-page.
However, there is more than one possible value for
$h_1 \cdot \tau^2 h_0$; it could be any linear combination of
$\rho (\tau h_1)^2$ and $\rho^5 h_2^2$.
We know that $\rho \cdot \tau^2 h_0$ is zero,
while $h_2^2$ is $\rho$-local.  Therefore, we deduce that
$h_1 \cdot \tau^2 h_0$ equals $\rho (\tau h_1)^2$.
This type of $\rho$-local analysis allows us to nail down the precise
value of each hidden extension in every case where 
there is more than one possible non-zero value.
\end{proof}

\begin{prop}
\label{prop:ExtR-relations}
Table \ref{tab:ExtR-reln}
gives some relations in $\Ext_\R$ that are hidden in the
$\rho$-Bockstein spectral sequence.
These relations, together with the products given in Table \ref{tab:ExtR-mult},
are a complete set of multiplicative
relations for $\Ext_\R$ up to Milnor-Witt degree $4$.
\end{prop}

\begin{proof}
These relations follow from the same types of arguments that are given
in the proof of Lemma \ref{lem:ExtR-mult}.
The most interesting is the relation
$h_0^2 \cdot \tau^2 h_2 + (\tau h_1)^3 = \rho^5 c_0$,
which follows from an analysis of the matric Massey product
\[
\left\langle
\rho^2,
\left[
\begin{array}{cc}
h_0 & \tau h_1 
\end{array} 
\right],
\left[
\begin{array}{c}
h_0 \cdot \tau^2 h_2 \\
(\tau h_1)^2 
\end{array}
\right]
\right\rangle.
\]
If this matric Massey product were defined, then 
May's Convergence Theorem and the differential 
$d_2(\tau^2) = \rho^2 \tau h_1$ would imply that it 
is detected by $\tau^4 h_1^2$ in the $\rho$-Bockstein
$E_\infty$-page. 
But $\tau^4 h_1^2$ does not survive to the $\rho$-Bockstein $E_\infty$-page.

For every monomial $x$ in Table \ref{tab:ExtR-gen} and every multiplicative
generator $y$ of $\Ext_\R$, one can check by brute force that
the relations in Tables \ref{tab:ExtR-mult} and \ref{tab:ExtR-reln}
allow one to identify
$x y$ in terms of the monomials in Table \ref{tab:ExtR-gen}.
\end{proof}

Figure \ref{fig:ExtR} displays $\Ext_\R$, sorted by Milnor-Witt
degree.  The picture is similar to the $E_\infty$-page shown in
Figure \ref{fig:Bockstein}, except that the hidden extensions by
$h_0$ and by $h_1$ are indicated with dashed lines.


\section{The Adams spectral sequence}
\label{sctn:Adams}

At this point  we have computed the tri-graded ring
\[ \Ext_\R = \Ext_{\cA}^{*,*,*}(\M_2,\M_2) 
\]
up through Milnor-Witt degree four.  We will now consider the motivic
Adams spectral sequence based on mod $2$ motivic cohomology, which
takes the form
\[ \Ext_{\cA}^{s,f,w}(\M_2,\M_2) \Rightarrow
\hat\pi_{s,w}. \]
Recall that we are writing $\hat\pi_{*,*}$ for the motivic
stable homotopy groups of the completion of the motivic sphere spectrum
with respect to the motivic Eilenberg-Mac Lane spectrum $H \F_2$.
The Adams $d_r$ differential takes elements of tridegree $(s,f,w)$ to
elements of tridegree $(s-1,f+r,w)$.
In particular, the Adams $d_r$ differential decreases the
Milnor-Witt degree by $1$.
So it pays off to once again fracture the $E_2$-page into the
different Milnor-Witt degrees. 

It turns out that there are no Adams differentials in the
range under consideration, as shown in the following result.

\begin{prop}
\label{prop:ASS-diffs}
Up through Milnor-Witt degree four, there are no differentials in the
motivic Adams spectral sequence.
\end{prop}

\begin{proof}
The proof uses Table \ref{tab:ExtR-gen}
and the $\Ext_\R$ charts in Figure \ref{fig:ExtR} to keep track of elements.

The elements $\rho$, $h_0$, and $h_1$ are permanent cycles, 
as there are no classes in Milnor-Witt degree $-1$. 
For $\tau^2 h_0$, we observe that 
there are no classes of Milnor-Witt degree $1$ in the range of the possible
differentials on $\tau^2 h_0$.
Similarly,
there are no possible values in Milnor-Witt degree $2$ for differentials
on $\tau h_2^2$, $h_3$, and $c_0$.

For degree reasons, the only possible values for
$d_r (\tau h_1)$ are $h_0^{r+1}$ and $\rho^{r+1} h_1^{r+1}$.
However, $h_0^2 \cdot h_0^{r+1}$ is non-zero on the
Adams $E_r$-page, while $h_0^2 \cdot \tau h_1$ is zero.
Also, $\rho^2 \cdot \rho^{r+1} h_1^{r+1}$ is non-zero
on the Adams $E_r$-page, while $\rho^2 \cdot \tau h_1$ is zero.
This implies that there are no differentials
on $\tau h_1$.

The only possible value for
$d_r(h_2)$ is $\rho^{r-1} h_1^{r+1}$.
However, $h_1 \cdot \rho^{r-1} h_1^{r+1}$ is non-zero on the
Adams $E_r$-page, while $h_1 \cdot h_2$ is zero.
This implies that there are no differentials on $h_2$.

The only possibility for a nonzero differential on $\tau^4 h_0$ is
that $d_2(\tau^4 h_0)$ might equal to $\rho^{10} h_1^2 h_3$.
However, $\rho \cdot \tau^4 h_0$ is zero on the Adams $E_2$-page,
while $\rho \cdot \rho^{10} h_1^2 h_3$ is not.
This implies that there are no differentials on $\tau^4 h_0$.

It remains to show that $\tau^2 h_2$ and $\tau c_0$ are permanent cycles.
We handle these more complicated arguments below in 
Lemmas \ref{lem:t^2h2} and \ref{lem:tc0}.
\end{proof}

\begin{lemma}
\label{lem:p^2-th1-h2}
The Massey product
$\langle \rho^2, \tau h_1, h_2 \rangle$ contains $\tau^2 h_2$, 
with indeterminacy generated by $\rho^4 h_3$.
\end{lemma}

\begin{proof}
Apply May's Convergence Theorem \cite{May69},
using the $\rho$-Bockstein differential $d_2(\tau^2) = \rho^2 \cdot \tau h_1$.
This shows that $\tau^2 h_2$ or $\tau^2 h_2 + \rho^4 h_3$ is contained
in the bracket.  
By inspection, the indeterminacy is generated by $\rho^4 h_3$.
\end{proof}

\begin{lemma}
\label{lem:t^2h2}
The element $\tau^2 h_2$ is a permanent cycle.
\end{lemma}

\begin{proof}
As shown in Table \ref{tab:notation},
let $\tau \eta$ and $\nu$ be elements of $\hat\pi_{1,0}$ and 
$\hat\pi_{3,2}$ respectively that are detected by
$\tau h_1$ and $h_2$.
The product $\rho^2 \cdot \tau \eta$ is zero because there is no other
possibility. For degree reasons, 
the product $\tau \eta \cdot \nu$ could possibly equal $\rho^2 \nu^2$.
However, $\rho^2 \cdot \tau \eta \cdot \nu$ is zero, while
$\rho^2 \cdot \rho^2 \nu^2$ is not.
Therefore,
$\tau \eta \cdot \nu$ is also zero.

We have just shown that the Toda bracket $\langle \rho^2, \tau \eta, \nu \rangle$
is well-defined.
Moss's Convergence Theorem \cite{Moss70} then implies that
the Massey product $\langle \rho^2, \tau h_1, h_2 \rangle$
contains a permanent cycle.
We computed this Massey product in
Lemma \ref{lem:p^2-th1-h2}, so we know that
$\tau^2 h_2$ or $\tau^2 h_2 + \rho^4 h_3$ is a permanent cycle.
We already know that $\rho^4 h_3$ is a permanent cycle, so 
$\tau^2 h_2$ is also a permanent cycle.
\end{proof}

For completeness, we will give an alternative proof
that $\tau^2 h_2$ is a permanent cycle that has a more geometric 
flavor.
There is a functor from classical homotopy theory to motivic
homotopy theory over $\R$ (or over any field) that takes
the sphere $S^p$ to $S^{p,0}$.
Let $\nu_{\tp}$ 
be the unstable map $S^{7,0} \map S^{4,0}$ that is
the image under this functor of the classical Hopf map $S^7 \map S^4$.

\begin{lemma}
\label{lem:Cnutop}
The cohomology of the cofiber of $\nu_{\tp}$
is a free $\M_2$-module on two generators $x$ and $y$ of degrees
$(4,0)$ and $(8,0)$, satisfying $\Sq^4(x) = \tau^2 y$ and 
$\Sq^8(x) = \rho^4 y$.
\end{lemma}

\begin{proof}
Consider the cofiber sequence
\[
S^{7,0} \map S^{4,0} \map C\nu_{\tp} \map S^{8,0},
\]
where $C\nu_{\tp}$ is the cofiber of $\nu_{\tp}$.
Apply motivic cohomology to obtain a long exact sequence.
It follows that
the cohomology of $C\nu_{\tp}$ is a
free $\M_2$-module on two generators $x$ and $y$ of degrees
$(4,0)$ and $(8,0)$.

For degree reasons, the only possible non-zero cohomology operations
are that $\Sq^4(x)$ and $\Sq^8(x)$ might equal
$\tau^2 y$ and $\rho^4 y$ respectively.
The formula $\Sq^4(x) = \tau^2 y$ follows by
comparison to the classical case.

The formula for $\Sq^8(x)$ is more difficult.
Consider
$S^{4,4} \Smash C\nu_{\tp}$, which has cells in dimensions
$(8,4)$ and $(12,8)$.
The cohomology generator in degree $(8,4)$ is the external
product $z \Smash x$, where $z$ is the cohomology
generator of $S^{4,4}$ in degree $(4,4)$.
The cohomology generator in degree $(12,4)$ is $z \Smash y$.

Now we can compute $\Sq^8$ in terms of the cup product
$(z \Smash x)^2$.
According to \cite{Voevodsky03a}*{Lemma 6.8},
the cup product $z^2$ equals $\rho^4 z$ in the
cohomology of $S^{4,4}$.
Also, the cup product $x^2$ equals $y$ in the cohomology
of $C\nu_{\tp}$ by comparison to the classical case.
By the K\"unneth formula, it follows that 
$(z \Smash x)^2 = \rho^4 (z \Smash y)$ and that
$\Sq^8(x) = \rho^4 y$.
\end{proof}

\begin{proof}[Another proof of Lemma \ref{lem:t^2h2}]
Lemma \ref{lem:Cnutop} shows that
the stabilization of $\nu_{\tp}$ in $\hat\pi_{3,0}$ is detected by
$\tau^2 h_2 + \rho^4 h_3$ in the motivic Adams spectral sequence.
There are elements $\rho$ and $\sigma$
in $\hat\pi_{-1,-1}$ and $\hat\pi_{7,4}$ detected
by $\rho$ and $h_3$ in the motivic Adams spectral sequence.
Therefore,
$\tau^2 h_2$ is a permanent cycle that detects $\nu_{\tp} + \rho^4 \sigma$.
\end{proof}

\begin{lemma}
\label{lem:th1-h2-h0h2}
The Massey product
$\langle \tau h_1, h_2, h_0 h_2 \rangle$ contains $\tau c_0$,
with indeterminacy generated by $\rho \cdot \tau h_1 \cdot h_1 h_3$.
\end{lemma}

\begin{proof}
Recall that there is a classical Massey product
$\langle h_1, h_2, h_0 h_2 \rangle = c_0$.
This implies that the motivic Massey product
$\langle \tau h_1, h_2, h_0 h_2 \rangle$ contains $\tau c_0$.

By inspection, the indeterminacy is generated by
$\rho \cdot \tau h_1 \cdot h_1 h_3$.
\end{proof}

\begin{lemma}
\label{lem:tc0}
The element $\tau c_0$ is a permanent cycle.
\end{lemma}

\begin{proof}
As shown in Table \ref{tab:notation}, let 
$\tau \eta$ and $\omega$ be elements of $\hat\pi_{1,0}$ and $\hat\pi_{0,0}$
detected by $\tau h_1$ and $h_0$.
As in the proof of Lemma \ref{lem:t^2h2},
the product $\tau \eta \cdot \nu$ is zero.
Also, the product $\omega \nu^2$ is zero because there are no other
possibilities.

We have just shown that the 
Toda bracket $\langle \tau \eta, \nu, \omega \nu \rangle$ is well-defined.
Moss's Convergence Theorem \cite{Moss70} then implies that
the Massey product $\langle \tau h_1, h_2, h_0 h_2 \rangle$
contains a permanent cycle.
We computed this Massey product in
Lemma \ref{lem:th1-h2-h0h2}, so we know that
$\tau c_0$ or $\tau c_0 + \rho \cdot \tau h_1 \cdot h_1 h_3$ is a permanent cycle.
We already know that $\rho \cdot \tau h_1 \cdot h_1 h_3$ is a permanent cycle, 
so 
$\tau c_0$ is also a permanent cycle.
\end{proof}

\section{Milnor-Witt modules and $\hat\pi_{*,*}$}
\label{sctn:MW-module}

In this section, we will describe how to pass from
the Adams $E_\infty$-page to $\hat\pi_{*,*}$.
We recall the following well-known elements
\cite{DI13} \cite{Morel04}.
\begin{enumerate}
\item
$\epsilon$ in $\hat\pi_{0,0}$ is represented by the twist map
on $S^{1,1} \Smash S^{1,1}$.
\item
$\rho$ in $\hat\pi_{-1,-1}$ is represented by the inclusion
$\{ \pm 1 \} \map (\A^1-0)$.
\item
$\eta$ in $\hat\pi_{1,1}$ is represented by the Hopf construction on the
multiplication $(\A^1 -0) \times (\A^1-0) \map (\A^1 -0)$.
\item
$\nu$ in $\hat\pi_{3,2}$ is represented by the Hopf construction on a version
of quaternionic multiplication.
\item
$\sigma$ in $\hat\pi_{7,4}$ is represented by the Hopf construction on
a version of octonionic multiplication.
\end{enumerate}

The element $1-\epsilon$ is detected in the Adams spectral sequence by $h_0$.
Thus it, rather than $2$, deserves to be considered the zeroth motivic Hopf
map.  Because it
plays a critical role, it is convenient to
give this element a name.

\begin{defn}
\label{defn:omega}
The element $\omega$ in $\hat\pi_{0,0}$ equals $1 - \epsilon$.
\end{defn}

The motivic Adams $E_\infty$-page is the 
associated graded module of the motivic homotopy groups
$\hat\pi_{*,*}$
with respect to the adic filtration for the ideal
generated by $\omega$ and $\eta$.
Note that this ideal also equals $(2,\eta)$
because of the relation $\rho \eta = - 1 - \epsilon = \omega - 2$.

The elements $\rho$, $h_0$, and $h_1$ detect the homotopy elements
$\rho$, $\omega$, and $\eta$ in the $0$th Milnor-Witt stem $\Pi_0$.
The relation $2 = \omega - \rho \eta$ implies that
$h_0 + \rho h_1$, rather than $h_0$, detects $2$.  This means
that we must be careful when computing the additive structure
of Milnor-Witt stems.

In the Adams chart,
a parallelogram such as
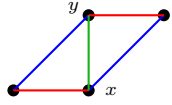
\begin{center}
\begin{pspicture}(0,-0.5)(1,1.5)
\pscircle*(0,0){\cirrad}
\pscircle*(1,0){\cirrad}
\pscircle*(1,1){\cirrad}
\pscircle*(2,1){\cirrad}
\psline[linecolor=honecolor](0,0)(1,1)
\psline[linecolor=honecolor](1,0)(2,1)
\psline[linecolor=pcolor](0,0)(1,0)
\psline[linecolor=pcolor](1,1)(2,1)
\psline[linecolor=hzerocolor](1,0)(1,1)
\rput(1.3,0){$\scriptstyle{x}$}
\rput(0.8,1.1){$\scriptstyle{y}$}
\end{pspicture}
\end{center}
indicates that $2$ times the homotopy elements detected by $x$
are zero (or detected in higher Adams filtration by a hidden
extension) because $(h_0 + \rho h_1) x = 0$.
On the other hand, a parallelogram such as
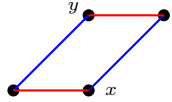
\begin{center}
\begin{pspicture}(0,-0.5)(1,1.5)
\pscircle*(0,0){\cirrad}
\pscircle*(1,0){\cirrad}
\pscircle*(1,1){\cirrad}
\pscircle*(2,1){\cirrad}
\psline[linecolor=honecolor](0,0)(1,1)
\psline[linecolor=honecolor](1,0)(2,1)
\psline[linecolor=pcolor](0,0)(1,0)
\psline[linecolor=pcolor](1,1)(2,1)
\rput(1.3,0){$\scriptstyle{x}$}
\rput(0.8,1.1){$\scriptstyle{y}$}
\end{pspicture}
\end{center}
indicates that $2$ times the homotopy elements detected by
$x$ are detected by $y$ because
$(h_0 + \rho h_1) x = y$.

We will choose specific homotopy elements to serve
as our $\Pi_0$-module generators.  Because of the associated graded
nature of the Adams $E_\infty$-page, there is some choice in these
generators.  
For the most part, the $\Pi_0$-module structures of 
the Milnor-Witt modules in our range
are insensitive to these choices,
so this is not of immediate concern.  
However, we would like to be as precise as we can
to facilitate further study.

These observations allow us to 
pass from the Adams spectral sequence to 
the diagrams of $\Pi_0$, $\Pi_1$, $\Pi_2$, and $\Pi_3$
given in Figure \ref{fig:MW-module}.

\subsection{The zeroth Milnor-Witt module}

For $\Pi_0$, the Adams spectral sequence consists
of an infinite sequence of dots extending upwards
in each stem except for the $0$-stem.
These
dots are all connected by $2$ extensions, so they assemble into
copies of $\Z_2$.
The 0-stem is somewhat more complicated.  Here there are two
sequences of dots extending upwards: elements of the form
$\rho^k h_1^k$ and elements of the form
$(h_0 + \rho h_1)^k = h_0^k + \rho^k h_1^k$.
The former elements assemble into a copy of
$\Z_2$ generated by $\rho \eta$, while the
latter elements assemble into a copy of $\Z_2$ generated by $1$.

\subsection{The first Milnor-Witt module}

For $\Pi_1$,
there are three elements in the 3-stem of the Adams spectral sequence. 
These elements assemble into a copy of $\Z/8$
generated by $\nu$; note that $h_0^k h_2 = (h_0 + \rho h_1)^k h_2$
because $h_1 h_2 = 0$.
The two elements $\tau h_1$ and $\rho \tau h_1^2$ 
in the 1-stem do not assemble into a copy
of $\Z/4$ because $(h_0 + \rho h_1) \tau h_1$ is zero.

We will now discuss precise definitions of
the $\Pi_0$-module generators of $\Pi_1$.

Recall that there is a functor from classical homotopy theory
to motivic homotopy theory over $\R$ that takes a sphere $S^{p}$
to $S^{p,0}$.  Let $\eta_{\tp}$ in $\hat\pi_{1,0}$ be the image
of the classical Hopf map $\eta$.
By an argument analogous to the proof of Lemma \ref{lem:Cnutop},
$\eta_{\tp}$ is detected by $\tau h_1 + \rho^2 h_2$.
Therefore
$\eta_{\tp} + \rho^2 \nu$ is detected by $\tau h_1$.

\begin{defn}
Let $\tau \eta$ be the element $\eta_{\tp} + \rho^2 \nu$ of $\hat\pi_{1,0}$.
\end{defn}

Another possible approach to defining $\tau \eta$ is to use a Toda bracket
to specify a single element.  However, the obvious Toda brackets detecting
$\tau h_1$ all have indeterminacy, so they are unsuitable for this purpose.

In terms of algebraic formulas we could write
\[ \Pi_1=\Pi_0 \langle \tau \eta, \nu \rangle /
(2\cdot \tau \eta, 8 \nu, \eta \nu, \rho^2 \cdot \tau \eta,
\eta^2 \cdot \tau \eta - 4 \nu ),
\]
but we find Figure \ref{fig:MW-module} to be more informative.

\subsection{The second Milnor-Witt module}

The calculation of $\Pi_2$ involves the same kinds of considerations
that we already described for $\Pi_0$ and for $\Pi_1$.
The names of the generators $(\tau \eta)^2$ and $\nu^2$ reflect the
multiplicative structure of the Milnor-Witt stems.

It remains to specify a choice of generator detected by $\tau^2 h_0$.
Recall that there is a realization functor from motivic homotopy theory
over $\R$ to classical homotopy theory.  (This functor factors through
$\Z/2$-equivariant homotopy theory, but we won't use the equivariance
for now.)

\begin{defn}
Let $\tau^2 \omega$ be the element of $\hat\pi_{0,-2}$
detected by $\tau^2 h_0$ that realizes to $2$ in classical $\pi_0$.
\end{defn}

In terms of algebraic formulas we could write
\[ \Pi_2 = \Pi_0 \langle \nu^2 \rangle /
( 2 \nu^2, \eta \nu^2) \oplus
\Pi_0 \langle \tau^2 \omega, (\tau \eta)^2 \rangle /
(\rho \cdot \tau^2 \omega, 2(\tau \eta)^2, \eta^2 (\tau \eta)^2,
\rho(\tau \eta)^2-\eta \cdot \tau^2 \omega).
\]
The unreadability of this formula illustrates why the graphical calculus 
of Figure \ref{fig:MW-module} is so helpful.

\subsection{The third Milnor-Witt module}

The structure of $\Pi_3$ is significantly more complicated.
We will begin by discussing choices of generators.

\begin{defn}
Let $\tau^2 \nu$ be the element $\nu_{\tp} + \rho^4 \sigma$
in $\hat\pi_{3,0}$.
\end{defn}

A precise definition 
of the generator $\tau \nu^2$ detected by $\tau h_2^2$
has so far eluded us.
For the purposes of this article, it suffices to choose
an arbitrary homotopy element detected by $\tau h_2^2$.
The distinction between the choices is not relevant in the 
range under consideration here, but it may be important 
for an analysis of higher Milnor-Witt stems.

Lemma \ref{lem:epsilonbar-defn}
gives a definition of the generator
detected by $c_0$, assuming that 
$\tau \nu^2$ has already been chosen.

\begin{lemma}
\label{lem:epsilonbar-defn}
There is a unique element $\epsilonbar$ in $\hat\pi_{8,5}$
detected by $c_0$ such that $\rho \epsilonbar - \eta \cdot \tau \nu^2$
equals zero.
\end{lemma}

\begin{proof}
Let $\epsilonbar'$ be any element detected by $c_0$.
The relation $\rho c_0 + h_1 \cdot \tau h_2^2 = 0$
implies that $\rho \epsilonbar' - \eta \cdot \tau \nu^2$
is detected in higher Adams filtration.
Note that $\omega$ kills
$\rho \epsilonbar' - \eta \cdot \tau \nu^2$ because
it kills both $\rho$ and $\eta$.
Therefore,
$\rho \epsilonbar' - \eta \cdot \tau \nu^2$ cannot be
detected by $h_0^3 h_3$.
If 
$\rho \epsilonbar' - \eta \cdot \tau \nu^2$ is detected by
$\rho^2 h_1 c_0$, then we can add an element
detected by $\rho h_1 c_0$ to $\epsilonbar'$ to obtain our desired element
$\epsilonbar$.
Similarly, if
$\rho \epsilonbar' - \eta \cdot \tau \nu^2$ is detected by
$\rho^3 h_1^2 c_0$, then we can add an element
detected by $\rho^2 h_1^2 c_0$ to $\epsilonbar'$.
\end{proof}

\begin{remark}
The classical analogue of $\epsilonbar$ is traditionally
called $\epsilon$; we have changed the notation to avoid the 
unfortunate coincidence with the motivic element
$\epsilon$ in $\hat\pi_{0,0}$.
\end{remark}

Having determined generators for $\Pi_3$,
we now proceed to analyze its $\Pi_0$-module structure.
For the most part, this analysis follows the same arguments
familiar from the earlier Milnor-Witt stems. 
The 3-stem and the 7-stem present the greatest challenges,
so we discuss them in more detail.

In the 3-stem,
$\tau^2 \nu$ generates a copy of $\Z/8$;
these eight elements are detected by
$\tau^2 h_2$,
$h_0 \cdot \tau^2 h_2 + \rho^3 \cdot \tau h_2^2$,
and $h_0^2 \cdot \tau^2 h_2 + \rho^5 c_0$.
The element $\rho^3 \cdot \tau \nu^2$
also generates a copy of $\Z/8$;
these eight elements are detected by
$\rho^3 \cdot \tau h_2^2$,
$\rho^5 c_0$, and $\rho^6 h_1 c_0$.
Finally, $\rho^4 \sigma$ generates a copy of $\Z/8$;
these eight elements are detected by
$\rho^4 h_3$, $\rho^5 h_1 h_3$,
and $\rho^6 h_1^2 h_3$.

In the 7-stem, $\sigma$ generates a copy of $\Z/32$;
these 32 elements are detected by
$h_3$,
$h_0 h_3 + \rho h_1 h_3$,
$h_0^2 h_3 + \rho^2 h_1^2 h_3$,
$h_0^3 h_3$,
and $\rho^3 h_1^2 c_0$.
The element $\rho \eta \sigma$ generates a copy of $\Z/4$;
these four elements are detected by
$\rho h_1 h_3$ and $\rho^2 h_1^2 h_3$.
The element $\rho \epsilonbar - 4 \sigma$ also generates
a copy of $\Z/4$; these four elements are detected
by $h_0^2 h_3 + \rho^2 h_1^2 h_3 + \rho c_0$ and
$h_0^3 h_3 + \rho^2 h_1 c_0$.

The $\eta$ extension from $\eta^2 \sigma$ to
$\eta^2 \epsilonbar$ is hidden in the Adams spectral sequence.
This is the same as the analogous hidden extension in the classical
situation \cite{Toda62}.  This is the only hidden extension by
$\rho$, $\omega$, or $\eta$ in the range under consideration.


\newpage

\section{Tables}
\label{sctn:table}

\begin{longtable}{lll}
\caption{Multiplicative generators for the $\rho$-Bockstein $E_1$-page} \\
\toprule
$s-w$ & element & $(s,f,w)$ \\
\midrule \endfirsthead
\caption[]{Multiplicative generators for the $\rho$-Bockstein $E_1$-page} \\
\toprule
$s-w$ & element & $(s,f,w)$ \\
\midrule \endhead
\bottomrule \endfoot
\label{tab:MW-gens}
$0$ & $\rho$ & $(-1,0,-1)$ \tabularnewline
$0$ & $h_0$ & $(0,1,0)$ \tabularnewline
$0$ & $h_1$ & $(1,1,1)$ \tabularnewline
$1$ & $\tau$ & $(0,0,-1)$ \tabularnewline
$1$ & $h_2$ & $(3,1,2)$ \tabularnewline
$3$ & $h_3$ & $(7,1,4)$ \tabularnewline
$3$ & $c_0$ & $(8,3,5)$ \tabularnewline
$4$ & $Ph_1$ & $(9,5,5)$ \tabularnewline
$5$ & $Ph_2$ & $(11,5,6)$ \tabularnewline
\end{longtable}

\begin{longtable}{lllll}
\caption{Bockstein differentials} \\
\toprule
$(s,f,w)$ & $x$ & $d_r$ & $d_r(x)$ & proof \\
\midrule \endfirsthead
\bottomrule \endfoot
\label{tab:Bockstein-diff}
$(0,0,-1)$ & $\tau$ & $d_1$ & $\rho h_0$ & Lemma \ref{lem:d1} \\
$(0,0,-2)$ & $\tau^2$ & $d_2$ & $\rho^2 \tau h_1$ & Lemma \ref{lem:d2} \\
$(0,0,-4)$ & $\tau^4$ & $d_4$ & $\rho^4 \tau^2 h_2$ & Lemma \ref{lem:d4} \\
$(1,1,-3)$ & $\tau^4 h_1$ & $d_6$ & $\rho^6 \tau h_2^2$ & Lemma \ref{lem:d6} \\
$(2,2,-2)$ & $\tau^4 h_1^2$ & $d_7$ & $\rho^7 c_0$ & Lemma \ref{lem:d6} \\
$(7,4,3)$ & $\tau h_0^3 h_3$ & $d_4$ & $\rho^4 h_1^2 c_0$ & Lemma \ref{lem:d4} \\
$(9,5,5)$ & $P h_1$ & $d_3$ & $\rho^3 h_1^3 c_0$ & Lemma \ref{lem:d3} \\
\end{longtable}

\begin{longtable}{lllllll}
\caption{$\F_2[\rho]$-module generators for the $\rho$-Bockstein $E_1$-page} \\
\toprule
 $0$ & $1$ & $2$ & $3$ & & $4$ & \\
\midrule \endfirsthead
\caption[]{$\F_2[\rho]$-module generators for the $\rho$-Bockstein $E_1$-page} \\
\toprule
$0$ & $1$ & $2$ & $3$ & & $4$ & \\
\midrule \endhead
\bottomrule \endfoot
\label{tab:rho-Bock-gens}
$h_0^k$ & $\tau$  & $\tau^2$ & $\tau^3$ & $h_3$ & $\tau^4$ & $\tau h_0 h_3$ \\ 
$h_1^k$ & $\tau h_0^k$ & $\tau^2h_0^k$ & $\tau^3h_0^k$ & $h_0 h_3$ & 
  $\tau^4h_0^k$ & $\tau h_0^2 h_3$ \\ 
& $\tau h_1$ & $\tau^2 h_1$ & $\tau^3h_1$ & $h_0^2 h_3$ & $\tau^4 h_1$ & 
  $\tau h_0^3 h_3$ \\ 
& $\tau h_1^2$ & $\tau^2 h_1^2$ & $\tau^3h_1^2$ & $h_0^3 h_3$ & $\tau^4 h_1^2$ &
  $\tau h_1 h_3$ \\
& $\tau h_1^3$ & $\tau^2 h_1^3$ & $\tau^3h_1^3$ & $h_1 h_3$ & $\tau^4 h_1^3$ &
  $\tau h_1^2 h_3$ \\
& $h_2$ & $\tau h_2$ & $\tau^2 h_2$ & $h_1^2 h_3$ & $\tau^3 h_2$ & $\tau c_0$ \\
& $h_0h_2$ & $\tau h_0h_2$ & $\tau^2 h_0h_2$ & $c_0$ & $\tau^3 h_0h_2$ & 
  $\tau h_1 c_0$ \\
& & $h_2^2$ & $\tau h_2^2$ & $h_1^k c_0$ & $\tau^2 h_2^2$ & $P h_1$ \\
& & & & & $\tau h_3$ & $h_1^k P h_1$ \\
\end{longtable}

\newpage

\begin{longtable}{llllll}
\caption{Some $\F_2[\rho]$-module generators for the $\rho$-Bockstein $E_2$-page} \\
\toprule
 $0$ & $1$ & $2$ & $3$ & $4$ & \\ 
\midrule \endfirsthead
\caption[]{Some $\F_2[\rho]$-module generators for the $\rho$-Bockstein $E_2$-page} \\
\toprule
 $0$ & $1$ & $2$ & $3$ & $4$ & \\ 
\midrule \endhead
\bottomrule \endfoot
\label{tab:rho-Bock-E2}
 & $\tau h_1$  & $\tau^2$ & $\tau^3h_1$ &
  $\tau^4$ & $\tau h_1 h_3$ \\ 
& $\tau h_1^2$ & $\tau^2 h_1$ & $\tau^3h_1^2$
   & $\tau^4 h_1$ & $\tau h_1^2 h_3$ \\ 
&  & $\tau^2 h_1^2$ & $\tau^2h_2$ 
   & $\tau^4 h_1^2$ & $\tau c_0$ \\ 
&  & $\tau^2 h_1^3$ & $\tau h_2^2$ &  
  $\tau^4 h_1^3$ & $\tau h_1 c_0$ \\
&  &  & $c_0$ 
   & $\tau^2 h_2^2$ & $P h_1$ \\
&  &  & $h_1^k c_0$  & $\tau h_0^3 h_3$ &
  $h_1^k P h_1$ \\
\end{longtable}

\begin{longtable}{llllll}
\caption{Some $\F_2[\rho]$-module generators for the $\rho$-Bockstein $E_3$-page} \\
\toprule
$0$ & $1$ & $2$ & $3$ & $4$ & \\
\midrule \endfirsthead 
\caption[]{Some $\F_2[\rho]$-module generators for the $\rho$-Bockstein $E_3$-page} \\
\toprule
$0$ & $1$ & $2$ & $3$ & $4$ & \\
\midrule \endhead 
\bottomrule \endfoot
\label{tab:rho-Bock-E3}
&  &  & $\tau^2 h_2$
   &  $\tau^4$ & $\tau h_0^3 h_3$ \\
&  &  & $\tau h_2^2$ 
  &  $\tau^4 h_1$ & $\tau c_0$ \\ 
&  &  &  $c_0$
   & $\tau^4 h_1^2$ & $P h_1$ \\
&  &  &  $h_1^k c_0$ & $\tau^4 h_1^3$ & $h_1^k P h_1$ \\
&  & &  &  $\tau^2 h_2^2$ &  \\
\end{longtable}

\begin{longtable}{llllll}
\caption{$\F_2[\rho]$-module generators for the $\rho$-Bockstein 
$E_\infty$-page} \\
\toprule
$0$ & $1$ & $2$ & $3$ & & $4$ \\
\midrule \endfirsthead 
\caption[]{$\F_2[\rho]$-module generators for the $\rho$-Bockstein 
$E_\infty$-page} \\
\toprule
$0$ & $1$ & $2$ & $3$ & & $4$ \\
\midrule \endhead 
\bottomrule \endfoot
\label{tab:rho-Bock-Einfty}
$h_0^k\,(\rho)$ & $\tau h_1\, (\rho^2)$  & $\tau^2 h_0^k\,(\rho)$ & 
  $\tau^3 h_1^3\,(\rho)$ & $h_0^3 h_3\,(\rho)$ & $\tau^4h_0^k\,(\rho)$ \\
$h_1^k$(loc) & $\tau h_1^2\, (\rho^2)$ & $\tau^2h_1^2\,(\rho^2)$ & 
  $\tau^2 h_2\,(\rho^4)$ & $h_1 h_3$(loc) & $\tau^2 h_2^2\,(\rho^4)$ \\
& $\tau h_1^3\,(\rho)$ & $\tau^2 h_1^3\,(\rho^2)$ & $\tau^2 h_0 h_2\,(\rho)$ 
  & $h_1^2 h_3$(loc) & $\tau h_1 h_3\,(\rho^2)$ \\
& $h_2$(loc) & $h_2^2$(loc) & $\tau h_2^2\,(\rho^6)$ &
  $c_0\,(\rho^7)$ & $\tau h_1^2 h_3\,(\rho^2)$ \\
& $h_0h_2\,(\rho)$ &  & $h_3$(loc) & $h_1 c_0\,(\rho^7)$ & $\tau c_0\,(\rho^3)$\\
&  & & $h_0 h_3\,(\rho)$ & $h_1^2 c_0\,(\rho^4)$ & $\tau h_1 c_0\,(\rho^2)$ \\
&  &  & $h_0^2 h_3\,(\rho)$ & $h_1^{k+3} c_0\,(\rho^3)$ &  \\
\end{longtable}

\begin{longtable}{llllll}
\caption{$\F_2[\rho]$-module generators for $\Ext_\R$} \\
\toprule
$0$ & $1$ & $2$ & $3$ & & $4$ \\
\midrule \endfirsthead
\caption[]{$\F_2[\rho]$-module generators for $\Ext_\R$} \\
\toprule
$0$ & $1$ & $2$ & $3$ & & $4$ \\
\midrule \endhead
\bottomrule \endfoot
\label{tab:ExtR-gen}
$h_0^k\,(\rho)$ & $\tau h_1\, (\rho^2)$  & $\tau^2 h_0\,(\rho)$ & 
  $\tau^2 h_2\,(\rho^4)$ & $h_0^3 h_3\,(\rho)$ & $\tau^4h_0\,(\rho)$ \\
$h_1^k$(loc) & $\tau h_1 \cdot h_1\, (\rho^2)$ 
   & $\tau^2 h_0 \cdot h_0^k\,(\rho)$ & 
  $h_0 \cdot \tau^2 h_2\,(\rho)$ & $h_1 h_3$(loc) 
  & $\tau^4 h_0 \cdot h_0^k\,(\rho)$ \\
& $h_2$(loc) & $(\tau h_1)^2\,(\rho^2)$ & $h_0^2 \cdot \tau^2 h_2\,(\rho)$ 
  & $h_1^2 h_3$(loc) & $\tau^2 h_2 \cdot h_2\,(\rho^4)$ \\
& $h_0 h_2\,(\rho)$ & $(\tau h_1^2) h_1\,(\rho^2)$ & $\tau h_2^2\,(\rho^6)$ &
  $c_0\,(\rho^7)$ & $\tau h_1 \cdot h_3\,(\rho^2)$ \\
& $h_0^2 h_2\,(\rho)$ & $h_2^2$(loc) & $h_3$(loc) & $h_1 c_0\,(\rho^7)$ 
   & $\tau h_1 \cdot h_1 h_3\,(\rho^2)$\\
&  & & $h_0 h_3\,(\rho)$ & $h_1^2 c_0\,(\rho^4)$ & $\tau c_0\,(\rho^3)$ \\
&  &  & $h_0^2 h_3\,(\rho)$ & $h_1^{k+3} c_0\,(\rho^3)$ & 
   $h_1 \cdot \tau c_0\,(\rho^2)$ \\
\end{longtable}

\begin{longtable}{llll}
\caption{Multiplicative generators of $\Ext_{\R}$} \\
\toprule
$(s,f,w)$ & generator & ambiguity & definition \\
\midrule \endfirsthead
\caption[]{$\Ext_{\R}$ generators} \\
\toprule
$(s,f,w)$ & generator \\
\midrule \endhead
\bottomrule \endfoot
\label{tab:Ext-gen}
$(0,1,0)$ & $h_0$ & $\rho h_1$ & $\rho \cdot h_0 = 0$ \\
$(1,1,1)$ & $h_1$ \\
$(1,1,0)$ & $\tau h_1$ & $\rho^2 h_2$ & $\rho^2 \cdot \tau h_1 = 0$ \\
$(3,1,2)$ & $h_2$ \\
$(0,1,-2)$ & $\tau^2 h_0$ \\
$(3,1,0)$ & $\tau^2 h_2$ & $\rho^4 h_3$ & $\rho^4 \cdot \tau^2 h_2 = 0$ \\
$(6,2,3)$ & $\tau h_2^2$ & $\rho^2 h_1 h_3$ & $\rho^6 \cdot \tau h_2^2 = 0$ \\
$(7,1,4)$ & $h_3$ \\
$(8,3,5)$ & $c_0$ & $\rho h_1^2 h_3$ & $\rho^7 \cdot c_0 = 0$ \\
$(0,1,-4)$ & $\tau^4 h_0$ \\
$(8,3,4)$ & $\tau c_0$ \\
\end{longtable}

\begin{longtable}{l|lllll}
\caption{$\Ext_\R$ multiplication table} \\
\toprule
\endfirsthead
\caption[]{$\Ext_\R$ multiplication table} \\
\toprule
\midrule \endhead
\bottomrule \endfoot
\label{tab:ExtR-mult}
& $h_0$ & $h_1$ & $\tau h_1$ & $h_2$ & $\tau^2 h_0$  \\
\midrule
$h_0$ & $-$ \\
$h_1$ & $0$ & $-$ \\
$\tau h_1$ & $\rho h_1 \cdot \tau h_1$ & $-$ & $-$ \\
$h_2$ & $-$ & $0$ & $0$ & $-$ \\
$\tau^2 h_0$ & $-$ & $\rho (\tau h_1)^2$ & $\rho^5 \tau h_2^2$ & $h_0 \cdot \tau^2 h_2$ & $\tau^4 h_0 \cdot h_0$ \\
$\tau^2 h_2$ & $-$ & $\rho^2 \tau h_2^2$ & $\rho^2 \tau^2 h_2 \cdot h_2$ & $-$ \\
$\tau h_2^2$ & $0$ & $\rho c_0$ & $\rho \tau c_0$ & $\tau h_1 \cdot h_1 h_3$ \\
$h_3$ & $-$ & $-$ & $-$ & $0$ \\
$c_0$ & $0$ & $-$ & $h_1 \cdot \tau c_0$ & $0$  \\
$\tau^4 h_0$ & $-$ & $0$ \\
$\tau c_0$ & $\rho h_1 \cdot \tau c_0$ & $-$  \\
\end{longtable}

\begin{longtable}{ll}
\caption{Some relations in $\Ext_\R$} \\
\toprule
\endfirsthead
\caption[]{Some relations in $\Ext_\R$} \\
\toprule
\midrule \endhead
\bottomrule \endfoot
\label{tab:ExtR-reln}
$(s,f,w)$ & relation \\
\midrule
$(3,3,2)$ & $h_0^2 h_2 + \tau h_1 \cdot h_1^2 = 0$ \\
$(3,3,0)$ & $h_0^2 \cdot \tau^2 h_2 + (\tau h_1)^3 = \rho^5 c_0$ \\
$(9,3,6)$ & $h_1^2 h_3 + h_2^3 = 0$ \\
$(6,3,4)$ & $h_0 h_2^2 = 0$ \\
$(3,4,0)$ & $h_0^3 \cdot \tau^2 h_2 = \rho^6 h_1 c_0$ \\
$(7,5,4)$ & $h_0^4 h_3 = \rho^3 h_1^2 c_0$ \\
$(6,3,2)$ & $h_0 h_2 \cdot \tau^2 h_2 = \rho^2 \cdot \tau c_0$ \\
$(10,5,6)$ & $h_1^2 \cdot \tau c_0 = 0$ \\
\end{longtable}

\begin{longtable}{llll}
\caption{Notation for $\hat\pi_{*,*}$} \\
\toprule
$(s,w)$ & element & $\Ext$ & definition \\
\midrule \endfirsthead
\caption[]{Notation for $\hat\pi_{*,*}$} \\
\toprule
$(s,w)$ & element & $\Ext$ & definition \\
\midrule \endhead
\bottomrule \endfoot
\label{tab:notation}
$(-1,-1)$ & $\rho$ & $\rho$ & $\{ \pm 1 \} \map (\A^1-0)$ \\
$(0,0)$ & $\epsilon$ & $1$ & twist on $S^{1,1} \Smash S^{1,1}$ \\
$(0,0)$ & $\omega$ & $h_0$ & $1-\epsilon$ \\
$(1,1)$ & $\eta$ & $h_1$ & Hopf construction \cite{DI13} \cite{Morel04} \\
$(1,0)$ & $\tau \eta$ & $\tau h_1$ & $\eta_\tp + \rho^2 \nu$ \\
$(3,2)$ & $\nu$ & $h_2$ & Hopf construction \cite{DI13} \\
$(0,-2)$ & $\tau^2 \omega$ & $\tau^2 h_0$ & realizes to $2$ \\
$(3,0)$ & $\tau^2 \nu$ & $\tau^2 h_2$ & $\nu_\tp + \rho^4 \sigma$ \\
$(6,3)$ & $\tau \nu^2$ & $\tau h_2^2$ \\
$(7,4)$ & $\sigma$ & $h_3$ & Hopf construction \cite{DI13} \\
$(8,5)$ & $\ol{\epsilon}$ & $c_0$ & $\rho \epsilonbar = \eta \cdot \tau \nu^2$ \\
\end{longtable}


\newpage

\section{Charts}
\label{sctn:chart}

This section contains the charts necessary to carry out the 
computations of the article.

Figure \ref{fig:ExtC} depicts $\Ext_\C$ in a range.  This data is lifted
directly from \cite{DI10} or
\cite{Isaksen14a}.  Here is a key for reading Figure \ref{fig:ExtC}:
\begin{enumerate}
\item
Black dots indicate copies of $\M_2^\C$.
\item
Red dots indicate copies of $\M_2^\C/\tau$.
\item
Lines indicate multiplications by $h_0$, $h_1$, and $h_2$.
\item
Red arrows indicate infinitely many copies of $\M_2^\C/\tau$
that are connected by $h_1$ multiplications.
\item
Magenta lines indicate that a multiplication hits $\tau$ times
a generator.  For example, $h_0 \cdot h_0 h_2$ equals $\tau h_1^3$.
\end{enumerate}

Figure \ref{fig:Bockstein} depicts the various pages of the
$\rho$-Bockstein spectral sequence,
sorted by Milnor-Witt degree.  See Section \ref{sctn:Bockstein-analysis}
for further discussion.
Here is a key for reading
Figure \ref{fig:Bockstein}:
\begin{enumerate}
\item
Black dots indicate copies of $\F_2$.
\item
Red lines indicate multiplications by $\rho$.
\item
Green lines indicate multiplications by $h_0$.
\item
Blue lines indicate multiplications by $h_1$.
\item
Red (or blue) arrows indicate infinitely many copies of $\F_2$ that are
connected by $\rho$ (or $\eta$) multiplications.
\end{enumerate}

Figure \ref{fig:ExtR} depicts $\Ext_\R = \Ext_\cA (\M_2, \M_2)$,
sorted by Milnor-Witt degree.  The key for this figure is
the same as for Figure \ref{fig:Bockstein}, with the additional:
\begin{enumerate}
\item
Dashed lines indicate $h_0$ or $h_1$ multiplications that
are hidden in the $\rho$-Bockstein spectral sequence.
\end{enumerate}
See Section \ref{sctn:ExtR} for discussion of these hidden extensions.

Figure \ref{fig:MW-module} depicts the Milnor-Witt
modules $\Pi_0$, $\Pi_1$, $\Pi_2$, and $\Pi_3$.
See Section \ref{sctn:MW-module} for further discussion.
Here is a key for reading the diagram:
\begin{enumerate}
\item
Open circles denote copies of $\Z_2$.
\item
Solid dots denote copies of $\Z/2$.
\item
Open circles with an $n$ inside denote copies of $\Z/n$.
\item
Blue lines depict $\eta$-multiplications going to the right.
\item
Red lines depict $\rho$-multiplications going to the left.
\item
Lines labelled $n$ indicate that the
result of the multiplication is $n$ times the
labelled generator: for example, $\rho\cdot \rho\eta=-2\rho$ or
$\eta\cdot \rho^3=-2\rho^2$ in $\Pi_0$.
\item
Two blue (or red lines) with the same source indicate that the
multiplication by $\eta$ (or $\rho$) equals a linear combination.
For example, $\eta \cdot \tau \nu^2$ equals 
$(\rho \epsilonbar - 4 \sigma) + 4 \sigma = \rho \epsilonbar$
in $\Pi_3$.
\item
Arrows pointing off the diagram indicate infinitely many multiplications
by $\rho$ or by $\eta$.
\item
Elements in the same topolgical stem are aligned vertically.
For example, $\eta^3$, $\nu$, $\eta (\tau \eta)^2$, and
$\rho^3 \nu^2$ all belong to the 3-stem. Their
weights are 1, 2, 1, and 1 respectively; this can be deduced from
their stems and Milnor-Witt degrees.
\end{enumerate}

\begin{landscape}

\begin{myfigure}
\label{fig:ExtC}

\begin{center}
\textsc{Figure \ref{fig:ExtC}:} 
$\Ext_\C = \Ext_{\cA^\C}(\M_2^\C, \M_2^\C)$
\end{center}

\psset{unit=0.75}


\end{myfigure}

\end{landscape}

%
\psset{unit=0.41cm}
\psset{linewidth=0.05}

\begin{landscape}

\begin{myfigure}
\label{fig:Bockstein}

\begin{center}
\textsc{Figure \ref{fig:Bockstein}:} $\rho$-Bockstein spectral sequence
\end{center}


%

\newpage

\begin{center}
\textsc{Figure \ref{fig:Bockstein}:} $\rho$-Bockstein spectral sequence
(continued)
\end{center}


%

\newpage

\begin{center}
\textsc{Figure \ref{fig:Bockstein}:} $\rho$-Bockstein spectral sequence
(continued)
\end{center}


%

\end{myfigure}


\end{document}